\documentclass[journal]{IEEEtran}

\usepackage{amsmath,amssymb}
\usepackage{url}
\usepackage{caption}
\usepackage{multirow,booktabs}
\usepackage{mathtools}
\usepackage{hyperref}
\usepackage{color}
\usepackage{subfigure,balance}
\usepackage{cite}
\usepackage{graphicx}
\usepackage{epstopdf}
\usepackage[font=small,labelfont=bf]{caption}

\newcommand{\tr}{{{\mathsf T}}}

\newtheorem{definition}{Definition}
\newtheorem{theorem}{Theorem}

\newtheorem{proposition}{Proposition}
\newtheorem{remark}{Remark}
\newtheorem{lemma}{Lemma}
\newtheorem{example}{Example}

\newtheorem{corollary}{Corollary}

\newcommand{\DD}{\mbox{\it DD}}
\newcommand{\SOS}{\mbox{\it SOS}}
\newcommand{\SDD}{\mbox{\it SDD}}

\newcommand{\SDSOS}{\mbox{\it SDSOS}}

\newcommand{\cFW}{\mathcal{FW}}

\newcommand{\qed}{\hfill$\square$}

\newenvironment{proof}{~~\textit{Proof:}}{\hfill$\square$}

\begin{document}
\title{Block Factor-width-two Matrices and Their Applications to Semidefinite and Sum-of-squares Optimization}

\author{Yang Zheng,~\IEEEmembership{Member,~IEEE,}
        Aivar~Sootla, 
        and~Antonis~Papachristodoulou,~\IEEEmembership{Fellow,~IEEE}
\thanks{The first two authors contributed equally to this work. A preliminary version of part of this work appeared in~\cite{sootla2019block}. This work is supported by the EPSRC Grant EP/M002454/1.}
\thanks{Y. Zheng is with the Department of Electrical and Computer Engineering, University of California San Diego, CA 92093. (email: zhengy@eng.ucsd.edu)}
\thanks{A. Sootla and A. Papachristodoulou are with Department of Engineering Science, University of Oxford, Parks Road, Oxford, OX1 3PJ, U.K. (emails: \{aivar.sootla, antonis\}@eng.ox.ac.uk)}}

\maketitle

\begin{abstract}

Semidefinite and sum-of-squares (SOS) optimization are fundamental computational tools in many areas, including linear and nonlinear systems theory. However, the scale of problems that can be addressed reliably and efficiently is still limited. In this paper, we introduce a new notion of \emph{block factor-width-two matrices} and build a new hierarchy of inner and outer approximations of the cone of positive semidefinite (PSD) matrices. This notion is a block extension of the standard factor-width-two matrices, and allows for an improved inner-approximation of the PSD cone. In the context of SOS optimization, this leads to a block extension of the \emph{scaled diagonally dominant sum-of-squares (SDSOS)} polynomials. By varying a matrix partition, the notion of {block factor-width-two matrices} can balance a trade-off between the computation scalability and solution quality for solving semidefinite and SOS optimization problems. Numerical experiments on a range of large-scale instances confirm our theoretical findings.
\end{abstract}


\IEEEpeerreviewmaketitle

\section{Introduction}

\IEEEPARstart{S}{emidefinite} programs (SDPs) are a class of convex problems over the cone of positive semidefinite (PSD) matrices~\cite{vandenberghe1996semidefinite}, which {is} {one of the major computational tools} in linear control theory. Many analysis and synthesis problems in linear systems can be addressed via solving certain SDPs; see~\cite{boyd1994linear} for an overview. The later development of sum-of-squares (SOS) optimization~\cite{blekherman2012semidefinite, bernard2009moments} extends the applications of SDPs to nonlinear problems involving polynomials, and thus, allows addressing many nonlinear control problems systematically, \emph{e.g.}, certifying asymptotic stability of equilibrium points of nonlinear
systems~\cite{papachristodoulou2002construction,anderson2015advances}, approximating region of attraction~\cite{henrion2013convex,topcu2008local,chesi2011domain}, and providing bounds on infinite-time averages~\cite{fantuzzi2016bounds}.

\subsection{Motivation}
In theory, 
SDPs can be solved up to any arbitrary precision in polynomial time using second-order interior-point methods (IPMs)~\cite{vandenberghe1996semidefinite}. From a practical viewpoint, however,
the {computational} speed and reliability of the current SDP solvers becomes worse for many large-scale problems of practical interest. Consequently, developing fast and reliable SDP solvers for large-scale problems has received considerable attention in the literature. For instance, a general purpose first-order solver based on the alternating direction method of multipliers (ADMM) was developed in~\cite{o2016conic}. For SDP programs with chordal sparsity (a sparsity pattern modeled by chordal graphs~\cite{blair1993introduction}), fast ADMM-based algorithms were proposed in~\cite{zheng2019chordal}, and efficient IPMs were suggested in~\cite{andersen2010implementation,fukuda2001exploiting}.  Chordal sparsity in the context of SOS optimization was also exploited in~\cite{waki2006sums, zheng2018decomposition,zheng2019sparse}. The underlying idea in these sparsity exploiting approaches is to equivalently decompose a large sparse PSD constraint into a number of smaller PSD constraints, leading to significant computational savings for sparse problems.

{Since} the approaches in~\cite{zheng2019chordal, fukuda2001exploiting,andersen2010implementation, waki2006sums, zheng2018decomposition,zheng2019sparse} are only suitable {for sufficiently sparse} problems, an alternative approach to speed-up semidefinite and SOS optimization was proposed in~\cite{ahmadi2019dsos} for general SDPs, where the authors suggested to approximate the PSD cone $\mathbb{S}^n_+$ with the cone of factor-width-two matrices~\cite{boman2005factor}, denoted as $\mathcal{FW}^n_2$ ($n$ is the matrix dimension). A matrix has a factor-width two if it can be represented as a sum of PSD matrices of rank at most two~\cite{boman2005factor}, and thus it is also PSD. The cone of $\mathcal{FW}^n_2$ can be equivalently written as a number of second-order cone constraints, and thus linear optimization over $\mathcal{FW}^n_2$ can be addressed by a second-order cone program (SOCP), which is much more scalable in terms of memory requirements and time consumption compared to SDPs. This feature of scalability is demonstrated in a wide range of applications~\cite{ahmadi2019dsos}. We note that $\mathcal{FW}^n_2$ is the same as the set of {symmetric} scaled diagonally dominant (SDD) matrices~\cite{boman2005factor}, and the authors in~\cite{ahmadi2019dsos} adopted the terminology \emph{SDD} instead of \emph{factor-width-two}.

As already pointed out in~\cite{ahmadi2019dsos}, approximating the PSD cone $\mathbb{S}^n_+$ by the cone of factor-width-two matrices $\mathcal{FW}^n_2$ is conservative. Consequently, the restricted problem may be infeasible or the optimal solution of the program with $\mathcal{FW}^n_2$ may be significantly different from that of the original SDP. There are several approaches to bridge the gap between $\mathcal{FW}^n_2$ and $\mathbb{S}^n_+$, such as the basis pursuit algorithm in~\cite{ahmadi2015sum}. As discussed in~\cite[Section 5]{ahmadi2019dsos}, one may also employ the notion of factor-width-$k$ matrices (denoted as $\mathcal{FW}^n_k$) that can be decomposed into a sum of PSD matrices of rank at most $k$. However, enforcing this constraint is problematic due to a large number of $k \times k$ PSD constraints, which grows in a combinatorial fashion as $n$ or $k$ increases (e.g., when $k =3$, the number of small PSD constraints is already $\mathcal{O}(n^3)$). Therefore, the computational burden may actually increase using factor-width-$k$ matrices compared to the original SDP, while also being conservative. It is nontrivial to use factor-width-$k$ matrices to approximate SDPs in a practical way. We note that the authors in~\cite{blekherman2020sparse} have provided quantification for the approximation quality of the PSD cone using the dual of $\mathcal{FW}^n_k$ for any factor-width $k$.

\subsection{Contributions}
In this paper, we take a different approach to enrich the cone of factor-width-two matrices for the approximation of the PSD cone: we take inspiration from SDD matrices and consider their block extensions. Our key idea is to partition a matrix into a set of non-intersecting blocks of entries and to enforce SDD constraints on these blocks instead of the individual entries.
In this way, we can reduce the number of small blocks significantly compared to $\mathcal{FW}^n_k$ with $k \geq 3$, while still improving the approximation quality compared to $\mathcal{FW}^n_2$.
Precisely, the contributions of this paper are:
\begin{itemize}
    \item We introduce a new class of \emph{block factor-width-two} matrices, which can be decomposed into a sum of PSD matrices whose rank is bounded by the corresponding block sizes. One notable feature of block factor-width-two matrices is that they are less conservative than $\mathcal{FW}^n_2$ and more scalable than $\mathcal{FW}^n_k (k \geq 3)$. This new class of matrices  forms a proper cone, and via coarsening the partition, we can build a new hierarchy of inner and outer approximations of the PSD cone.

\item Motivated by~\cite{blekherman2020sparse}, we provide lower and upper bounds on the distance between the class of block factor-width-two matrices and the PSD cone after some normalization. Our results explicitly show that reducing the number of partitions can improve the approximation quality. This agrees with the results in~\cite{blekherman2020sparse} that require to increase the factor-width $k$. We highlight that reducing the number of partitions is numerically more efficient in practice since the number of decomposition bases is reduced as well. In addition, we identify a class of sparse PSD matrices that belong to the cone of block factor-width-two matrices.

 \item  We apply the notion of block factor-width-two matrices in both semidefinite and SOS optimization. We first define a new \emph{block factor-width-two cone program}, which is able to return an upper bound to the corresponding SDP faster. Then, in the context of SOS optimization, applying the notion of block factor-width-two matrices naturally leads to a block extension of the so-called \emph{SDSOS} polynomials~\cite{ahmadi2019dsos}. A new hierarchy of inner approximations of SOS polynomials is derived accordingly. We also show that a natural partition exists in the context of SOS matrices. Numerical tests from large-scale SDPs and SOS optimization show promising results in balancing a trade-off between computation scalability and solution quality using our notion of  {block factor-width-two} matrices.
\end{itemize}

\subsection{Related work}

Developing efficient and reliable methods to make SDPs scalable is a very active research area. There are extensive results in the literature, and we here overview some representative techniques on improving scalability for SDPs (see~\cite{de2010exploiting,majumdar2020recent,vandenberghe2015chordal,zheng2021chordal} for excellent surveys). One main class of approaches is to exploit problem structure (such as sparsity~\cite{vandenberghe2015chordal,zheng2019chordal} and symmetry~\cite{gatermann2004symmetry}) to enhance scalability. Another class of methods aims to generate low-rank solutions to SDPs which promises to reduce computational time and storage requirements; see the celebrated Burer-Monteiro algorithm~\cite{burer2003nonlinear}. Also, there exists increasing research attention on developing efficient first-order algorithms for SDPs which in general trade off scalability with accuracy~\cite{o2016conic,yang2015sdpnal,monteiro2014first}. 
An iterative algorithm based on the Augmented Lagrangian method has been developed in~\cite{kovcvara2003pennon}. Finally, another category of approaches is to impose structural approximations of the PSD cone and trade off scalability with conservatism. One typical technique is the aforementioned factor-width-two approximation~\cite{ahmadi2019dsos,ahmadi2015sum}. Our result on block factor-width-two matrices falls into the last category and extends the technique in~\cite{ahmadi2019dsos}. It will be interesting to combine different approaches for further scalability improvements of solving SDPs.

\subsection{Paper structure and notation}
The rest of this paper is organized as follows. In Section~\ref{section:preliminaries}, we briefly review some necessary preliminaries on matrix theory. Section~\ref{section:blockFW2} introduces the new class of block factor-width-two matrices and a new hierarchy of inner/outer approximations of the PSD cone. The approximation quality of block factor-width-two matrices is discussed in Section~\ref{section:approximation}. We present applications in semidefinite and SOS optimization in Section~\ref{section:Application}, and numerical experiments are reported in Section~\ref{section:Experiments}. We conclude the paper in Section~\ref{section:Conclusion}.

\emph{Notation:} Throughout this paper, we use $\mathbb{N} = \{1, 2, \ldots\}$ to denote the set of positive integers, and $\mathbb{R}$ to denote the set of real numbers. Given a matrix $A \in \mathbb{R}^{n\times n}$, we denote its transpose by $A^\tr$. We write $\mathbb{S}^n$ for the set of $n \times n$ symmetric matrices, and the set of $n\times n$ positive semidefinite (PSD) matrices is denoted as $\mathbb{S}^n_+$. When the dimensions are clear from the context, we also use $X \succeq 0$ to denote a PSD matrix. We use $I_{k}$ to denote an identity matrix of size $k \times k$, and $0$ to denote a zero block with appropriate dimensions that shall be clear from the context. A block-diagonal matrix with $D_1, \ldots, D_p$ on its diagonal entries is denoted as $\text{diag}(D_1, \ldots, D_p)$.

\section{Preliminaries} \label{section:preliminaries}

In this section, we present some preliminaries on matrix theory, including block-partitioned matrices, factor-width-$k$ matrices, and sparse PSD matrices.

\subsection{Block-partitioned matrices and two linear maps}

Given a matrix $A \in \mathbb{R}^{n \times n}$, we say a set of integers $\alpha = \{k_1, k_2, \ldots, k_p\}$ with $k_i \in \mathbb{N}\; (i = 1, \ldots, p)$  is a partition of matrix $A$ if $\sum_{i=1}^p k_i= n$, and $A$ is partitioned as
$$
    \begin{bmatrix}
        A_{11} & A_{12} & \ldots A_{1p} \\
        A_{21} & A_{22} & \ldots A_{2p} \\
        \vdots & \vdots & \ddots \vdots  \\
        A_{p1} & A_{p2} & \ldots A_{pp}
    \end{bmatrix},
$$
with $A_{ij} \in \mathbb{R}^{k_i \times k_j}, \forall i, j = 1, \ldots, p$. Throughout this paper, we assume the number of blocks in a partition is no less than two, i.e., $p \geq 2$. Obviously, a matrix $A \in \mathbb{R}^{n \times n}$ admits many partitions, and one trivial partition is $\alpha = \{1,1,\ldots,1\}$. We say $\alpha= \{k_1, \ldots, k_p\}$ is \emph{homogeneous}, if we have $k_i = k_j, \forall\,i, j = 1, \ldots, p$, where the matrix dimension satisfies $n = pk_1$. Next, we define a coarser/finer relation between two partitions $\alpha$ and $\beta$ for matrices in $\mathbb{R}^{n \times n}$.

\begin{definition}
    Given two partitions $\alpha = \{k_1, k_2, \ldots, k_p\}$ and $\beta = \{l_1, l_2, \ldots, l_q\}$ with $ p < q $ and $\sum_{i=1}^p k_i = \sum_{i=1}^q l_i$, we say $\beta$ is a sub-partition of $\alpha$, denoted as $\beta 	\sqsubseteq \alpha$, if there exist integers $\{m_1, m_2, \ldots, m_{p+1}\}$ with $m_1 = 1, m_{p+1} = q+1, m_i < m_{i+1}, i = 1, \ldots, p$ such that
    $$
        k_i = \sum_{j = m_i}^{m_{i+1} - 1} l_j, \; \forall i = 1, \ldots, p.
    $$
\end{definition}

Essentially, a subpartition of $\alpha =\{k_1, k_2, \ldots, k_p\}$ is a finer partition that breaks some blocks of $\alpha$ into smaller blocks. For example, given three partitions $\alpha = \{4, 2\}, \beta = \{2, 2, 2\}$ and $\gamma = \{1, 1, 1, 1, 1, 1\}$, we have
$
    \gamma \sqsubseteq \beta \sqsubseteq \alpha.
$
Fig.~\ref{FIG:1} illustrates these three partitions for a matrix in $\mathbb{R}^{6 \times 6}$. Given a partition $\alpha = \{k_1, \ldots, k_p\}$ with $\sum_{i=1}^p k_i = n$, we denote
$$
    E_{i}^{\alpha} = \begin{bmatrix} 0 &\ldots & I_{k_i} & \ldots & 0 \end{bmatrix} \in \mathbb{R}^{k_i \times n},
$$
which forms a partition of the identity matrix of size $n \times n$,
\begin{equation} \label{eq:indentityAlpha}
    I_n = \begin{bmatrix} I_{k_1} & & & \\
        & I_{k_2}  & & \\
        &  & \ddots & \\
        & & & I_{k_p} \\
                \end{bmatrix} = \begin{bmatrix} E_{1}^{\alpha} \\ E_{2}^{\alpha} \\
                \vdots\\ E_{p}^{\alpha}\\
                \end{bmatrix}.
\end{equation}
We also denote
\begin{equation} \label{eq:blockbasis}
        E^{\alpha}_{ij} = \begin{bmatrix} (E^{\alpha}_i)^\tr & (E^{\alpha}_j)^\tr \end{bmatrix}^\tr \in \mathbb{R}^{(k_i + k_j) \times n}, i \neq j.
\end{equation}
More generally, for a set of distinct indices $\mathcal{C} = \{i_1, \ldots, i_m\}$ and $1 \leq i_1 < \ldots < i_m \leq p$, we define
$$
    E^{\alpha}_{\mathcal{C}} = \begin{bmatrix} (E_{i_1}^{\alpha})^\tr & (E_{i_2}^{\alpha})^\tr & \ldots & (E_{i_m}^{\alpha})^\tr\end{bmatrix}^\tr \in \mathbb{R}^{|\mathcal{C}|\times n},
$$
where $|\mathcal{C}| = \sum_{i \in \mathcal{C}} k_{i} $. For a trivial partition $\alpha = \{1, 1, \ldots, 1\}$, notations $E_{i}^{\alpha}, E_{ij}^{\alpha}, E_{\mathcal{C}}^{\alpha}$ are simplified as $E_{i}, E_{ij}, E_{\mathcal{C}}$, respectively. Note that $E_i$ is the $i$-th standard unit vector in $\mathbb{R}^n$.

\begin{figure}
	\centering
	\subfigure[]{
		\includegraphics[scale=.5]{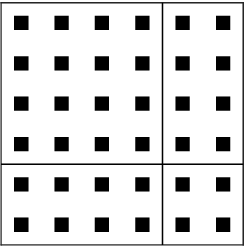}
	}
	\hspace{2mm}
	\subfigure[]{
		\includegraphics[scale=.5]{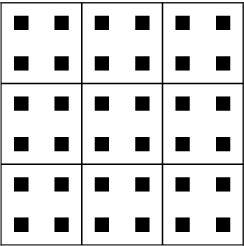}
	}
	\hspace{2mm}
	\subfigure[]{
		\includegraphics[scale=.5]{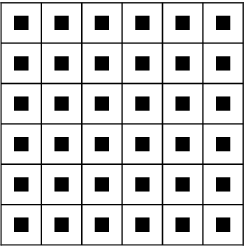}
	}
	\caption{Different partitions for a $6 \times 6$ matrix: (a) $\alpha=\{4, 2\}$, (b) $\beta = \{2, 2, 2\}$, (c) $\gamma = \{1, 1, 1, 1, 1, 1\}$. Here, each black square represents a real number. From right to left, we get coarser partitions, \emph{i.e.} $ \gamma \sqsubseteq \beta \sqsubseteq \alpha$.}
	\label{FIG:1}
\end{figure}

For a block matrix $A$ with partition $\alpha = \{k_1, \ldots, k_p\}$, the matrix $E_{\mathcal{C}}^{\alpha}$ with set $\mathcal{C} = \{i_1, \ldots,i_m\}$ can be used to define two linear maps:
\begin{itemize}
    \item  1) \emph{Truncation operator}, which selects a principle submatrix from $A$, \emph{i.e.},
$$
    Y := E_{\mathcal{C}}^{\alpha} A (E_{\mathcal{C}}^{\alpha})^\tr = \begin{bmatrix} A_{i_1i_1} &  \ldots  &A_{i_1i_m} \\
     A_{i_2i_1} & \ldots  &A_{i_2i_m} \\
      \vdots & \ldots  &\vdots  \\
      A_{i_mi_1} & \ldots  &A_{i_mi_m} \\\end{bmatrix} \in \mathbb{R}^{|\mathcal{C}| \times|\mathcal{C}|}.
$$
\item  2) \emph{Lift operator}, which creates an $n \times n$ matrix from a
matrix of dimension $|\mathcal{C}| \times |\mathcal{C}|$,  \emph{i.e.},
$
    (E_{\mathcal{C}}^{\alpha})^\tr Y E_{\mathcal{C}}^{\alpha}\in \mathbb{R}^{n \times n},
$ for a given matrix $Y \in \mathbb{R}^{|\mathcal{C}| \times|\mathcal{C}|}$.
\end{itemize}
Finally, we define a block permutation matrix with respect to a partition $\alpha$: consider an $n \times n$ identity matrix partitioned as~\eqref{eq:indentityAlpha}. A block $\alpha$-permutation matrix $P_{\alpha}$ is a matrix obtained by permuting the block-wise rows of $I_n$ in~\eqref{eq:indentityAlpha} according to some permutation of the numbers 1 to $p$. For instance, if $\alpha = \{k_1, k_2\}$, then $P_{\alpha}$ is in one of the following forms
$$
    \begin{bmatrix}
        I_{k_1} & \\
        &  I_{k_2}
    \end{bmatrix}, \quad \begin{bmatrix}
         &  I_{k_2}\\
        I_{k_1}&
    \end{bmatrix}.
 $$

\subsection{Factor-width-$k$ matrices}

We now introduce the concept of \emph{factor-width-$k$ matrices}, originally defined in~\cite{boman2005factor}.

\begin{definition}
    The factor width of a PSD matrix $X$ is the smallest integer $k$ such that there exists a matrix $V$ where $A = VV^\tr$ and each column of $V$ has at most $k$ non-zeros.
    \end{definition}

Equivalently, the factor-width of $X$ is the smallest integer $k$ for which $X$ can be written as the sum of PSD matrices that are non-zero only on a single $k \times k$ principal submatrix. We use $\mathcal{FW}_k^n$ to denote the set of $n \times  n$ matrices of factor-width no greater than $k$. Then, we have the following inner approximations of $\mathbb{S}^n_+$,
\begin{equation} \label{eq:fwkinner}
        \mathcal{FW}_1^n \subseteq \mathcal{FW}_2^n \subseteq \ldots \subseteq \mathcal{FW}_n^n = \mathbb{S}^n_+.
\end{equation}
It is not difficult to see that
$
    Z \in \mathcal{FW}^n_k
$ if and only if there exist $Z_i \in \mathbb{S}^k_+$ such that
\begin{equation}\label{eq:fwk}
    Z = \sum_{i=1}^s E_{\mathcal{C}_i}^\tr Z_i E_{\mathcal{C}_i},
\end{equation}
where $\mathcal{C}_i$ is a set of $k$ distinct integers from 1 to $n$ and $s = { n \choose k}$. We say~\eqref{eq:fwk} is a factor-wdith-$k$ decomposition of $Z$.

The dual of $\mathcal{FW}^n_k$ with respect to the trace inner product is
$$
    (\mathcal{FW}^n_k)^* = \left\{X \in \mathbb{S}^n \mid E_{\mathcal{C}_i}XE_{\mathcal{C}_i}^\tr \in \mathbb{S}^k_+, \forall i = 1, \ldots, s\right\}.
$$
Then, we also have a hierarchy of outer approximations of the PSD cone $\mathbb{S}^n_+$
$$
    \mathbb{S}^n_+ = (\mathcal{FW}_n^n)^* \subseteq \ldots \subseteq  (\mathcal{FW}_2^n)^* \subseteq (\mathcal{FW}_1^n)^*.
$$

Particularly, an interesting case is $\mathcal{FW}_2^n$, which is the same as the set of symmetric scaled diagonally dominant matrices~\cite{boman2005factor}. Linear optimization over $\mathcal{FW}_2^n$ can be equivalently converted into an SOCP, for which efficient algorithms exist. This feature of scalability is the main motivation of the so-called SDSOS optimization in~\cite{ahmadi2019dsos} that utilizes $\mathcal{FW}_2^n$. For completeness, the definition of scaled diagonally dominant matrices is given as follows.
\begin{definition}
    A symmetric matrix $A \in \mathbb{S}^n$ with entries $a_{ij}$ is diagonally dominant (DD) if
    $$
        a_{ii} \geq \sum_{j\neq i} |a_{ij}|, \forall i = 1, \ldots, n.
    $$
    A  symmetric matrix $A \in \mathbb{S}^n$  is scaled diagonally dominant (SDD) if there exists a diagonal matrix $D$ with positive diagonal entries such that $DAD$ is diagonally dominant.
\end{definition}

We denote the set of $n \times n$ DD and SDD matrices as $\DD_n$ and $\SDD_n$, respectively. It is not difficult to see that
$$
    \DD_n \subseteq \SDD_n \subseteq \mathbb{S}^n_+.
$$
Also, it is proved in~\cite{boman2005factor} that $SDD_n  = \mathcal{FW}^n_2$.

\subsection{Sparse PSD matrices}

This section covers some notation 
on sparse PSD matrices. Here, we use an undirected graph to describe the sparsity pattern of a symmetric matrix $X \in \mathbb{S}^n$ with partition $\alpha = \{k_1, k_2, \ldots, k_p\}$.
A graph $\mathcal{G}(\mathcal{V}, \mathcal{E})$ is defined by a set of vertices $\mathcal{V} = \{1, 2, \ldots , p\}$ and a set of
edges $\mathcal{E} \subseteq \mathcal{V}\times \mathcal{V}$. Here, we only consider graphs with no self-loops, i.e., $(i,i) \notin \mathcal{E}$. A graph $\mathcal{G}(\mathcal{V}, \mathcal{E})$ is undirected if $(i,j) \in \mathcal{E} \Rightarrow (j,i) \in \mathcal{E}$.

Given a partition $\alpha = \{k_1, k_2, \ldots, k_p\}$, we define a set of sparse block matrices defined by a graph $\mathcal{G}(\mathcal{V}, \mathcal{E})$ as
$$
    \mathbb{S}^n_{\alpha}(\mathcal{E},0) = \{X \in \mathbb{S}^n \mid X_{ij} = 0, \; \text{if}\; (i,j) \notin \mathcal{E}, i \neq j\},
$$
where $X_{ij} \in \mathbb{R}^{k_i \times k_j}$. The set of sparse block PSD matrices is defined as
$$
     \mathbb{S}^n_{\alpha,+}(\mathcal{E},0) = \mathbb{S}^n_{\alpha}(\mathcal{E},0)  \cap \mathbb{S}^n_{+},
$$
and the set of PSD completable matrices is defined as
$$
     \mathbb{S}^n_{\alpha,+}(\mathcal{E},?) = \mathbb{P}_{\mathbb{S}^n_{\alpha}(\mathcal{E},0)}\left(\mathbb{S}^n_{+}\right),
$$
where $\mathbb{P}_{\mathbb{S}^n_{\alpha}(\mathcal{E},0)}\left(\cdot \right)$ denotes a projection onto  the space of $\mathbb{S}^n_{\alpha}(\mathcal{E},0)$ with respect to the usual Frobenius matrix norm, i.e., it replaces the blocks outside $\mathcal{E}$ with zeros. For any undirected graph $\mathcal{G}(\mathcal{V}, \mathcal{E})$, the cones $ \mathbb{S}^n_{\alpha,+}(\mathcal{E},0)$ and $\mathbb{S}^n_{\alpha,+}(\mathcal{E},?)$  are dual to each other with respect to the trace inner product. For a trivial partition $\alpha = \{1, 1, \ldots, 1\}$, notations $ \mathbb{S}^n_{\alpha}(\mathcal{E},0)$,  $\mathbb{S}^n_{\alpha,+}(\mathcal{E},0)$, $\mathbb{S}^n_{\alpha,+}(\mathcal{E},?) $ are simplified as $ \mathbb{S}^n(\mathcal{E},0)$,  $\mathbb{S}^n_{+}(\mathcal{E},0)$,  $\mathbb{S}^n_{+}(\mathcal{E},?) $, respectively.

One important feature of $\mathbb{S}^n_{\alpha,+}(\mathcal{E},0)$  and $\mathbb{S}^n_{\alpha,+}(\mathcal{E},?) $ is that they allow an equivalent decomposition when the graph $\mathcal{G}(\mathcal{V}, \mathcal{E})$ is chordal. Recall that an undirected graph is called chordal if every cycle of length greater than three has at least one chord~\cite{blair1993introduction}. A chord is an edge that connects two non-consecutive nodes in a cycle (see~\cite{vandenberghe2015chordal} for details). Before introducing the decomposition of $\mathbb{S}^n_{\alpha,+}(\mathcal{E},0)$  and $\mathbb{S}^n_{\alpha,+}(\mathcal{E},?)$, we need to define another concept of cliques: a clique $\mathcal{C}$ is a subset of vertices where $(i,j) \in \mathcal{E}, \forall i, j \in \mathcal{C}$, and it is called a maximal clique if it is not contained in another clique.

\begin{theorem}[{\!\!\cite{agler1988positive,grone1984positive,zheng2019a}}] \label{prop:chordal}
    Given a chordal graph $\mathcal{G}(\mathcal{V},\mathcal{E})$ with maximal cliques $\mathcal{C}_1, \ldots, \mathcal{C}_g$ and a partition $\alpha = \{k_1, k_2, \ldots, k_p\}$, we have
    \begin{itemize}
        \item $X \in \mathbb{S}^n_{\alpha,+}(\mathcal{E},?)$ if and only if
        $
            E^{\alpha}_{\mathcal{C}_i}X(E^{\alpha}_{\mathcal{C}_i})^\tr \in \mathbb{S}^{|\mathcal{C}_i|}_+, i = 1, \ldots, g.
        $
        \item $Z \in \mathbb{S}^n_{\alpha,+}(\mathcal{E},0)$ if and only if there exist a set of matrices $Z_i \in \mathbb{S}^{|\mathcal{C}_i|}_+, i = 1, \ldots, g$, such that
        \begin{equation} \label{eq:chordal}
            Z = \sum_{i=1}^g (E^{\alpha}_{\mathcal{C}_i})^\tr Z_i (E^{\alpha}_{\mathcal{C}_i}).
        \end{equation}
    \end{itemize}
\end{theorem}
For the trivial partition $\alpha = \{1, 1, \ldots, 1\}$, Theorem~\ref{prop:chordal} was originally proved in~\cite{agler1988positive,grone1984positive}. The extension to an arbitrary partition  $\alpha = \{k_1, k_2, \ldots, k_p\}$ was given in~\cite[Chapter 2.4]{zheng2019a}. We note that this decomposition underpins many recent algorithms on exploiting sparsity in semidefinite programs; see~\emph{e.g.},~\cite{vandenberghe2015chordal,zheng2019chordal}.

\begin{remark}[Factor-width decomposition and sparse chordal decomposition]
It is clear that the factor-width decomposition~\eqref{eq:fwk} and the sparse chordal decomposition~\eqref{eq:chordal} are in the same decomposition form but with two distinctive differences: 1) the number of components is a combinatorial number ${n \choose k}$ in~\eqref{eq:fwk}, while the number is bounded by the number of maximal cliques in~\eqref{eq:chordal}; 2) the size of each component in~\eqref{eq:fwk} is fixed as the factor-width $k$ while the size is determined by the corresponding maximal clique in~\eqref{eq:chordal}. Note that $\mathcal{FW}^n_k$ is an inner approximation of $\mathbb{S}^n_{+}$ while the decomposition~\eqref{eq:chordal} is necessary and sufficient for the cone $\mathbb{S}^n_{\alpha,+}(\mathcal{E},0)$ with a chordal sparsity pattern $\mathcal{E}$. \hfill $\square$
\end{remark}

\section{Block factor-width-two matrices} \label{section:blockFW2}

Since there are a combinatorial number ${ n \choose k}$ of  smaller 
matrices of size $k \times k$, a complete parameterization of $\mathcal{FW}^n_k$ is not always practical using~\eqref{eq:fwk}. In other words, even though $\mathcal{FW}^n_k$ is an inner approximation of $\mathbb{S}^n_+$, it does not necessarily mean {that} checking the membership of  $\mathcal{FW}^n_k$ is always computationally cheaper than that of $\mathbb{S}^n_+$. For instance, optimizing over $\mathcal{FW}^n_3$ requires $\mathcal{O}(n^3)$ PSD constraints of size $3 \times 3$, which is prohibitive for even moderate $n$. It appears that the only practical case is $\mathcal{FW}^n_2$ which is the same as $SDD_n$, where we have
$$
    Z \in \mathcal{FW}^n_2 \Leftrightarrow Z = \sum_{i=1}^{n-1} \sum_{j = i+1}^n E_{ij}^\tr Z_{ij} E_{ij} \; \text{with}\; Z_{ij} \in \mathbb{S}^2_{+}.
$$
This constraint $Z \in \mathcal{FW}^n_2$ can be further reformulated into $\mathcal{O}(n^2)$ second-order cone constraints, for which efficient solvers are available. However, the gap between  $\mathcal{FW}^n_2$ and $\mathbb{S}^n_+$ might be unacceptable in some applications.

To bridge this gap, we introduce a new class of block factor-width-two matrices in this section. We show that this class of matrices is less conservative than $\mathcal{FW}^n_2$ and more scalable than $\mathcal{FW}^n_3 (k \geq 3)$ for the inner approximation of $\mathbb{S}^n_+$.

\subsection{Definition and a new hierarchy of inner/outer approximations of the PSD cone}

The class of block factor-width-two matrices is defined as follows.

\begin{definition}
    A symmetric matrix $Z \in \mathbb{S}^n$ with partition $\alpha = \{k_1, k_2, \ldots, k_p\}$ belongs to the class of block factor-width-two matrices, denoted as $\mathcal{FW}_{\alpha,2}^n$, if and only if 
    \begin{equation} \label{eq:BlkFW}
        Z = \sum_{i=1}^{p-1}\sum_{j=i+1}^p (E^{\alpha}_{ij})^\tr X_{ij} E^{\alpha}_{ij}
    \end{equation}
    for some $X_{ij} \in \mathbb{S}^{k_i+k_j}_+$ and with $E^{\alpha}_{ij}$ defined in~\eqref{eq:blockbasis}.
\end{definition}

\begin{figure}
	\centering
		\includegraphics[scale=.75]{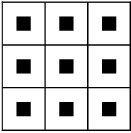}
		\raisebox{7.5mm}{$=$}
		\includegraphics[scale=.75]{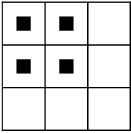}
		\raisebox{7.5mm}{$+$}
		\includegraphics[scale=.75]{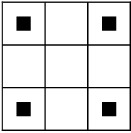}
		\raisebox{7.5mm}{$+$}
		\includegraphics[scale=.75]{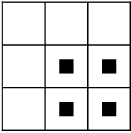}
	
	\caption{Block factor-width-two decomposition~\eqref{eq:BlkFW} for a PSD matrix with partition $\alpha =\{k_1, k_2, k_3\}$, where each summand is required to be PSD. The $(i,j)$ black square represents a submatrix of dimension $k_i \times k_j, i, j =1, 2, 3$.}
	\label{FIG:2}
\end{figure}

Fig.~\ref{FIG:2} demonstrates this definition for a PSD matrix with partition $\alpha = \{k_1, k_2, k_3\}$. This set of matrices has strong topological properties with an easy characterization of its dual cone, as shown below.

\begin{proposition}\label{prop:set-properties}
For any admissible $\alpha$, the dual of $\mathcal{FW}_{\alpha,2}^n$ with respect to the trace inner product is
$$
    (\mathcal{FW}_{\alpha,2}^n)^* = \{X \in \mathbb{S}^n \mid E^{\alpha}_{ij} X (E^{\alpha}_{ij})^\tr \succeq 0, 1\leq i < j \leq p\}.
$$
Furthermore, both $\mathcal{FW}_{\alpha,2}^n$ and $(\mathcal{FW}_{\alpha,2}^n)^*$ are proper cones, \emph{i.e.}, they are convex, closed,  solid, and  pointed cones.
\end{proposition}
\begin{proof}
The dual is computed by direct computation. Now, $\forall X_1, X_2 \in (\mathcal{FW}_{\alpha,2}^n)^*$ and  $\theta_1, \theta_2 \geq 0$, it is straightforward to verify
    $$
        \theta_1 X_1 + \theta_2 X_2 \in (\mathcal{FW}_{\alpha,2}^n)^*.
    $$
    Thus, $(\mathcal{FW}_{\alpha,2}^n)^*$ is a convex cone. The cone $(\mathcal{FW}_{\alpha,2}^n)^*$  is pointed because  $X \in (\mathcal{FW}_{\alpha,2}^n)^*, -X \in (\mathcal{FW}_{\alpha,2}^n)^* $ implies that $X = 0$. Also, $I_n \in (\mathcal{FW}_{\alpha,2}^n)^*$ is an interior point. It is closed because it can
be expressed as 
the intersection of infinite closed half spaces in $\mathbb{S}^n$: $X \in (\mathcal{FW}_{\alpha,2}^n)^*$ if and only if $$x_{ij}^\tr E^{\alpha}_{ij} X (E^{\alpha}_{ij})^\tr x_{ij} \geq 0, \forall x_{ij} \in \mathbb{R}^{k_i + k_j}, 1\leq i < j \leq p.$$
Note that each constant vector $x_{ij} \in \mathbb{R}^{k_i + k_j}$ defines a linear constraint on the variable $X$, corresponding to a half space in $\mathbb{S}^n$.
Now it is clear that $(\mathcal{FW}_{\alpha,2}^n)^*$ is proper. Therefore, the dual of $(\mathcal{FW}_{\alpha,2}^n)^*$ is $\mathcal{FW}_{\alpha,2}^n$, which is also proper.
\end{proof}

Besides these topological properties, our notion of block factor-width-two matrices offers a tuning mechanism to build hierarchies of these cones.
Intuitively, changing the matrix partition should allow a trade-off between approximation quality and scalability of computations. In particular, the following theorem is the main result of this section.

\begin{theorem}\label{theo:inclusion}
    Given three partitions $\alpha = \{k_1, \ldots, k_p\}, \beta = \{l_1, \ldots, l_q\}$ and $\gamma = \{n_1,n_2\}$, where $\sum_{i=1}^p k_i = \sum_{i=1}^q l_i =n_1 + n_2 = n$ and $\alpha  \sqsubseteq \beta$, we have the following inner approximations of $\mathbb{S}^n_+$:
    \begin{align*}
        \mathcal{FW}^n_2 = \mathcal{FW}_{\mathbf{1},2}^n \subseteq  \mathcal{FW}_{\alpha,2}^n \subseteq  \mathcal{FW}_{\beta,2}^n \subseteq  \mathcal{FW}_{\gamma,2}^n = \mathbb{S}^n_{+},
        \end{align*}
        as well as the following outer approximations of $\mathbb{S}^n_+$
        \begin{align*}
        \mathbb{S}^n_{+} = (\mathcal{FW}_{\gamma,2}^n)^* \subseteq (\mathcal{FW}_{\beta,2}^n)^* \subseteq (\mathcal{FW}_{\alpha,2}^n)^* \subseteq  (\mathcal{FW}_{\mathbf{1},2}^n)^*
    \end{align*}
    where $\mathbf{1} = \{1, \ldots, 1\}$ denotes the trivial partition.
\end{theorem}
\begin{proof}
     First, $\mathcal{FW}^n_2 = \mathcal{FW}_{\mathbf{1},2}^n$ and $\mathcal{FW}_{\gamma,2}^n = \mathbb{S}^n_{+}$ are true by definition. We only need to prove $\mathcal{FW}_{\alpha,2}^n \subset  \mathcal{FW}_{\beta,2}^n $ when $\alpha  \sqsubseteq \beta$, since  we always have $\mathbf{1} = \{1, \ldots, 1\} \sqsubseteq \alpha$ for a non-trivial partition $\alpha$.

    As we will show in Corollary~\ref{prop:invariant}, $\mathcal{FW}^n_{\alpha,2}$ is invariant with respect to block $\alpha$-permutation. Therefore,  to prove $\mathcal{FW}_{\alpha,2}^n \subset  \mathcal{FW}_{\beta,2}^n $ when $\alpha  \sqsubseteq \beta$, it is sufficient to consider the case
    \begin{equation} \label{eq:alpha_beta}
    \begin{aligned}
        \alpha &= \{k_1, \ldots, k_{p-1}, k_p, k_{p+1}\}, \\
        \beta  &= \{k_1, \ldots, k_{p-1}, k_p+k_{p+1}\},
    \end{aligned}
    \end{equation}
    where the partition $\beta$ is formed by merging the last two blocks in $\alpha$ and keeping the other blocks unchanged. All the other cases $\alpha  \sqsubseteq \beta$ can be formed recursively by combining the construction~\eqref{eq:alpha_beta} with some block $\alpha$-permutation.

    We now prove $\mathcal{FW}_{\alpha,2}^n \subset  \mathcal{FW}_{\beta,2}^n $ for~\eqref{eq:alpha_beta}. Our proof is constructive: for any $X \in \mathcal{FW}_{\alpha,2}^n$, we show that $X \in \mathcal{FW}_{\beta,2}^n$. Let $E^{\alpha}_{ij}, 1 \leq i < j \leq p+1$ be the decomposition bases for the $\alpha$-partition, and $E^{\beta}_{ij}, 1 \leq i < j \leq p$ be decomposition bases for the $\beta$-partition.  By definition~\eqref{eq:blockbasis}, we have
    $$
    E^{\beta}_{ij} = E^{\alpha}_{ij}, \qquad 1 \leq i < j \leq p-1,
    $$
    since the first $p-1$ blocks are the same for $\alpha$ and $\beta$.
    Given any $X \in \mathcal{FW}^n_{\alpha,2}$, there exist $X_{ij}\in \mathbb{S}^{k_i + k_j}_+$ such that
    \begin{equation} \label{eq:inclusion_s1}
                \begin{aligned}
            X &= \sum_{i=1}^p\sum_{j=i+1}^{p+1} (E^{\alpha}_{ij})^\tr X_{ij} E^{\alpha}_{ij} \\
            & = \sum_{i=1}^{p-1}\sum_{j=i+1}^{p-1} (E^{\alpha}_{ij})^\tr X_{ij} E^{\alpha}_{ij} + \sum_{i=1}^{p-1}(E^{\alpha}_{ip})^\tr X_{ip} E^{\alpha}_{ip}  \\
             &\qquad \qquad \qquad  +\sum_{i=1}^{p}(E^{\alpha}_{i(p+1)})^\tr X_{i(p+1)} E^{\alpha}_{i(p+1)}.
        \end{aligned}
    \end{equation}

    We proceed with constructing $\hat{X}_{ij}$ such that $X$ can be decomposed as
    \begin{equation}\label{eq:inclusion_s2}
         X = \sum_{i=1}^{p-1}\sum_{j=i+1}^{p} (E^{\beta}_{ij})^\tr \hat{X}_{ij} E^{\beta}_{ij}.
    \end{equation}
    Since the first $p-1$ blocks are the same in both partitions, we can choose
    \begin{equation} \label{eq:construction_hatX}
        \hat{X}_{ij} = X_{ij}, \qquad 1 \leq i < j \leq p -1.
    \end{equation}
    Comparing~\eqref{eq:inclusion_s1} with~\eqref{eq:inclusion_s2}, it remains to construct $\hat{X}_{ip}, \; i = 1, \ldots, p-1$ such that
    \begin{equation}\label{eq:inclusion_s3}
    \begin{aligned}
                \sum_{i=1}^{p-1} (E^{\beta}_{ip})^\tr \hat{X}_{ip} E^{\beta}_{ip} = & \sum_{i=1}^{p-1}(E^{\alpha}_{ip})^\tr X_{ip} E^{\alpha}_{ip}  \\ &+\sum_{i=1}^{p}(E^{\alpha}_{i(p+1)})^\tr X_{i(p+1)} E^{\alpha}_{i(p+1)}.
    \end{aligned}
    \end{equation}
    Consider the matrices $X_{ij} \in \mathbb{S}^{k_i + k_j}_+, 1 \leq i \leq p-1, j = p, p+1,$ in~\eqref{eq:inclusion_s3}, and we split them according to its partition
    $$
        X_{ij} = \begin{bmatrix} X_{ij,1} & X_{ij,2} \\
        \star & X_{ij,3}\end{bmatrix}
    $$
    with $X_{ij,1} \in \mathbb{S}^{k_i}_+, X_{ij,3} \in \mathbb{S}^{k_j}_+$ and $\star$ denoting the symmetric part. Then, based on some direct calculations, it can be verified that~\eqref{eq:inclusion_s3} holds when choosing $\hat{X}_{ip}, 1 \leq i \leq p-1$ as follows
    \begin{equation} \label{eq:construction_hatX_s2}
    \begin{aligned}
        &\hat{X}_{ip} = \frac{1}{p-1}\begin{bmatrix} 0 & 0 & 0 \\
                                 0 & X_{p(p+1),1} & X_{p(p+1),2} \\
                                 0 & \star & X_{p(p+1),3}\end{bmatrix} + \\
                               &\quad \begin{bmatrix} X_{i(p+1),1} & 0 & X_{i(p+1),2}  \\
                                 0 & 0 & 0 \\
                                 \star & 0 & X_{i(p+1),3}\end{bmatrix}\! + \!\begin{bmatrix} X_{ip,1} & X_{ip,2} & 0 \\
                                 \star & X_{ip,3} & 0 \\
                                 0 & 0 & 0\end{bmatrix}\!.
    \end{aligned}
    \end{equation}
    This completes the proof of the hierarchy of inner approximations using $\mathcal{FW}^{\alpha}_2$.

    Finally, the hierarchy of outer approximations using $(\mathcal{FW}_2^{\alpha})^*$ holds by standard duality arguments.
\end{proof}

We now compare the block factor-width-two matrices $\mathcal{FW}_{\alpha,2}^n$ and the standard factor-width $k$ matrices $\mathcal{FW}_{k}^n$. First, it is easy to notice that when $\alpha = \{k, \ldots, k\}$ and $n = kp$ (for which we call the partition is \emph{homogeneous}), we have
    \begin{equation*}
    \begin{aligned}
        \mathcal{FW}_{\alpha,2}^n \subseteq \mathcal{FW}_{2 k}^n,  \quad
        (\mathcal{FW}_{2 k}^n)^\ast \subseteq (\mathcal{FW}_{\alpha,2}^n)^\ast.
    \end{aligned}
    \end{equation*}
Second, both $\mathcal{FW}^n_{\alpha,2}$ and $\mathcal{FW}^n_{k}$ can be used to construct a hierarchy of inner/outer approximations of $\mathbb{S}^n_{+}$. One major difference lies in the number of basis matrices. In $\mathcal{FW}^n_{k}$, we need ${n \choose k}$ basis matrices for a complete parameterization, as shown in~\eqref{eq:fwk}, which is usually prohibitive in practice. Instead, in $\mathcal{FW}^n_{\alpha,2}$, we build a sequence of coarser partitions, and the number of basis matrices has been reduced to $\frac{p(p-1)}{2}$, which is more practical for numerical computation when the size of each block is moderate. Therefore, the cones $\cFW_{\alpha, 2}^n$ are often more scalable in terms of the number of variables (see Sections~\ref{section:Application} and~\ref{section:Experiments} for applications and experiments). 
\begin{figure}
	\centering
	\includegraphics[scale=.68]{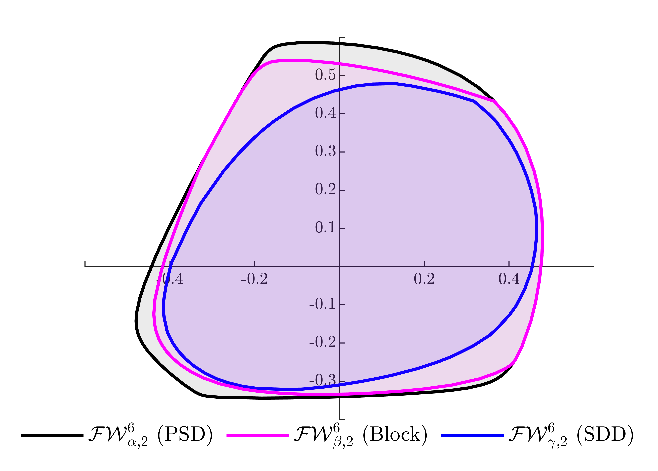}
	\caption{Boundry of the set of $x$ and $y$ for which the $6 \times 6$ symmetric matrix $    I_6 + xA + y B$ belongs to $\mathcal{FW}^6_{\alpha,2}, \mathcal{FW}^6_{\beta,2}$, and $\mathcal{FW}^6_{\gamma,2}$, where $\alpha=\{4, 2\}$, $\beta = \{2, 2, 2\}$, $\gamma = \{1, 1, 1, 1, 1, 1\}$. The relation $ \gamma \sqsubseteq \beta \sqsubseteq \alpha$ is reflected in the inclusion of $\SDD_6 =\mathcal{FW}^6_{\gamma,2} \subset \mathcal{FW}^6_{\beta,2} \subset  \mathcal{FW}^6_{\alpha,2}=\mathbb{S}^6_+$. }
	\label{FIG:3}
\end{figure}

 \begin{example}
We here illustrate the approximation quality of the cone $\cFW_{\alpha, 2}^n$ using {Fig.~\ref{FIG:3}}, where we plot the boundary of the set of $x$ and $y$ for which the $6 \times 6$ symmetric matrix
$$
    I_6 + xA + y B
$$
belongs to $\mathcal{FW}^6_{\alpha,2}, \mathcal{FW}^6_{\beta,2}$, and $\mathcal{FW}^6_{\gamma,2}$, where the partitions are the same as the example in Fig.~\ref{FIG:1}, \emph{i.e.}, $\alpha=\{4, 2\}$, $\beta = \{2, 2, 2\}$, $\gamma = \{1, 1, 1, 1, 1, 1\}$. In this case, $\mathcal{FW}^6_{\alpha,2}, \mathcal{FW}^6_{\gamma,2}$ are the same as PSD, and SDD, respectively. Here, the matrices $A$ and $B$ were generated randomly with independent and identically distributed entries sampled form the standard normal distribution. As expected by Theorem~\ref{theo:inclusion},  the relation $ \gamma \sqsubseteq \beta \sqsubseteq \alpha$ is reflected in the inclusion of $\SDD_6 =\mathcal{FW}^6_{\gamma,2} \subset \mathcal{FW}^6_{\beta,2} \subset  \mathcal{FW}^6_{\alpha,2}=\mathbb{S}^6_+$. \qed
 \end{example}

\begin{example} \label{example:factor_decomposition}
We consider another example to further illustrate Theorem~\ref{theo:inclusion}:
$$
    X = \begin{bmatrix} 6 & 8 & -2 & -2\\
    8 & 16 & 1 &1 \\
    -2 & 1 & 10 & -1\\
    -2& 1 & -1 & 24\end{bmatrix}.
$$
It can be verified that $X \in \mathcal{FW}^4_{\alpha,2}$ with partition $\alpha = \{1,1,1,1\}$, and the matrices in the decomposition~\eqref{eq:BlkFW} can be chosen as follows
$$
    \begin{aligned}
    X_{12} &= \begin{bmatrix} 4.5 & 8\\ 8 & 14.5 \end{bmatrix},
    X_{13} = \begin{bmatrix} 1 & -2\\ -2 & 6\end{bmatrix},
   X_{14} = \begin{bmatrix} 0.5 & -2\\ -2 & 12\end{bmatrix},
\\
    X_{23} &= \begin{bmatrix} 1 & 1\\ 1 & 2\end{bmatrix},  X_{24} = \begin{bmatrix} 0.5 & 1\\ 1 & 6\end{bmatrix}, X_{34} = \begin{bmatrix} 2 & -1\\ -1 & 6\end{bmatrix}.
    \end{aligned}
$$
Here, we note that the off-diagonal elements of $X_{ij}$ are the same with the corresponding off-diagonal elements of $X$. This fact motivates our alternative characterizations of $\mathcal{FW}^n_{\alpha,2}$ in Theorem~\ref{theo:blockfw}.

If we collapse the last two blocks into one single block and obtain a coarser partition $\beta = \{1,1,2\}$, then Theorem~\ref{theo:inclusion} confirms $X \in \mathcal{FW}^4_{\beta,2}$. Indeed, following the constructions in~\eqref{eq:construction_hatX} and~\eqref{eq:construction_hatX_s2}, we can choose $\hat{X}_{12} = X_{12}$ and obtain
$$
    \hat{X}_{13} = \begin{bmatrix} 1.5&-2&-2\\
    -2&7&-0.5\\
    -2&-0.5&15\end{bmatrix},  \hat{X}_{23} = \begin{bmatrix} 1.5&1&1\\
    1&3&-0.5\\
    1&-0.5&9\end{bmatrix}.
$$
\qed
\end{example}

\subsection{Another characterization and its corollaries}

As discussed in Example~\ref{example:factor_decomposition}, the decomposition matrices $X_{i j}$ use the actual off-diagonal values of the matrix $X$. This observation {allows us to derive an alternative description} of $\mathcal{FW}_{\alpha,2}^n$, offering a new interpretation of block factor-width-two matrices in terms of scaled diagonally dominance.
\begin{theorem} \label{theo:blockfw}
    Given a partition $\alpha = \{k_1, \ldots, k_p\}$ with $\sum_{i=1}^p k_i= n$, we have $A \in \mathcal{FW}_{\alpha,2}^n$ if and only if there exist $Z_{ij} \in \mathbb{S}^{k_i}_+, i, j = 1, \ldots, p, i \neq j$, such that
    \begin{subequations}
    \begin{align} 
        A_{ii} \succeq \sum_{j =1, j \neq i}^p Z_{ij}, &\,\,\forall\; i = 1, \ldots, p  \label{eq:blockfw2_s1}\\
        \begin{bmatrix} Z_{ij} & A_{ij} \\
         \star  & Z_{ji}\end{bmatrix} \succeq 0, &\,\,\forall \; 1 \leq i < j \leq p, \label{eq:blockfw2_s2}
    \end{align}
    \end{subequations}
    where $\star$ denotes the symmetric counterpart.
\end{theorem}
\begin{proof}
    $\Rightarrow:$ Suppose $A \in \mathcal{FW}_{\alpha,2}^n$. By definition, we have
    \begin{equation} \label{eq:blockfw}
        A = \sum_{i=1}^{p-1}\sum_{j=i+1}^p (E^{\alpha}_{ij})^\tr X_{ij} E^{\alpha}_{ij}
    \end{equation}
    for some $X_{ij} \in \mathbb{S}^{k_i+k_j}_+$. Let $X_{ij} \in \mathbb{S}^{k_i+k_j}_+$ in~\eqref{eq:blockfw} be partitioned as
    \begin{equation} \label{eq:blockfactor}
        X_{ij} = \begin{bmatrix} X_{ij,1} & X_{ij,2} \\ \star  & X_{ij,3} \end{bmatrix} \succeq 0,
    \end{equation}
    with $ X_{ij,1} \in \mathbb{S}^{k_i}_+, X_{ij,3} \in \mathbb{S}^{k_j}_+$. By construction, we know
    $$
    \begin{aligned}
        A_{ij} &= X_{ij,2}, \forall \; 1 \leq i < j \leq p, \\
        A_{ii} &= \sum_{i < j} X_{ij,1} + \sum_{i > j} X_{ji,3}, \forall \; i = 1, \ldots, p.
        \end{aligned}
    $$
    Now we set
    $$
        Z_{ij} =  \begin{cases}
           X_{ij,1}, \quad \text{if}\; i < j,\\
           X_{ji,3}, \quad \text{if}\; i > j,\\
        \end{cases}
    $$
    which naturally satisfy~\eqref{eq:blockfw2_s1} and~\eqref{eq:blockfw2_s2}.

    $\Leftarrow:$ Suppose we have~\eqref{eq:blockfw2_s1} and~\eqref{eq:blockfw2_s2}. We next construct $X_{ij}\in \mathbb{S}^{k_i+k_j}_+, 1 \leq i<j\leq p$ of the form~\eqref{eq:blockfactor} that satisfy~\eqref{eq:blockfw}. We first let
    $$
        Q_{ii} = {A_{ii} - \sum_{j=1, j\neq i}^p Z_{ij}} \succeq 0, i = 1, \ldots, p.
    $$
    Now, $\forall 1 \leq i<j\leq p$, we set
    $$
        X_{ij,1} = Z_{ij} + \frac{1}{p-1}Q_{ii}, \; X_{ij,2} = A_{ij}, \; X_{ij,3} = Z_{ji} + \frac{1}{p-1}Q_{jj}.
    $$
    Since we have~\eqref{eq:blockfw2_s2}, we know $X_{ij}\in \mathbb{S}^{k_i+k_j}_+, 1 \leq i<j\leq p$ are in the form~\eqref{eq:blockfactor}. Also, by construction, we have~\eqref{eq:blockfw} is satisfied. Thus, $A  \in \mathcal{FW}_{\alpha,2}^n$.
\end{proof}

For illustration, we remark that for a partition with two blocks, \emph{i.e.}, $\alpha = \{k_1, k_2\}$ with $k_1 + k_2 = n$, Theorem~\ref{theo:blockfw} simply enforces a PSD property on matrix $A$, \emph{i.e.},
$$
    A = \begin{bmatrix} A_{11} & A_{12} \\ * & A_{22}\end{bmatrix} \succeq 0
$$
if and only if there exists $Z_{12} \in \mathbb{S}^{k_1}_+, Z_{21} \in \mathbb{S}^{k_2}_+$ such that
$$
    A_{11} \succeq Z_{12}, \quad A_{22} \succeq Z_{21},  \quad \begin{bmatrix} Z_{12} & A_{12} \\ \star  & Z_{21}\end{bmatrix} \succeq 0.
$$
This representation for a $2\times 2$ block-partitioned matrix is just to illustrate~\eqref{eq:blockfw2_s1}-\eqref{eq:blockfw2_s2} in Theorem~\ref{theo:blockfw}, but we note that it is not useful for scalable numerical computation.

\begin{remark}[Block scaled diagonal dominance]
Theorem~\ref{theo:blockfw} shows that the class of block factor-width-two matrices can be considered as a block extension of the SDD matrices. It can be interpreted that the diagonal block $A_{ii}$ should dominate the sum of the off-diagonal blocks $A_{ij}$ in terms of positive semidefiniteness. \qed
\end{remark}

The conditions~\eqref{eq:blockfw2_s1} and~\eqref{eq:blockfw2_s2} were derived using a block generalization of the strategies for the SDD matrices in~\cite{sootla2017block,sootla2019existence}. Indeed,~\eqref{eq:blockfw2_s1} and~\eqref{eq:blockfw2_s2} reduce to the condition of scaled diagonal dominance in the trivial partition case, \emph{i.e.}, $\alpha = \{1, \ldots, 1\}$, $A = [a_{ij}] \in \mathbb{S}^n$. In this case,~\eqref{eq:blockfw2_s1} and~\eqref{eq:blockfw2_s2} become
 \begin{subequations}
    \begin{align} 
        a_{ii} &\geq \sum_{j =1, j \neq i}^n z_{ij},& &\forall \; i = 1, \ldots, n,  \label{eq:fw2_s1}\\
        |a_{ij}| &\leq \sqrt{z_{ij}z_{ji}}, & &\forall \; 1 \leq i < j \leq n, \label{eq:fw2_s2} \\
        z_{ij} &\geq 0, &  &\forall \; i, j = 1, \ldots, n, i \neq j. \label{eq:fw2_s3}
    \end{align}
    \end{subequations}
    We have the following result.
    \begin{corollary}\label{prop:fwsdd}
            Given a symmetric matrix $A  = [a_{ij}] \in \mathbb{S}^n$, the following statements are equivalent.
            \begin{enumerate}
               \item $A \in \mathcal{FW}^n_2$;
                \item             There exists $z_{ij} \geq 0$ satisfying~\eqref{eq:fw2_s1} -- \eqref{eq:fw2_s3};
            \item $A \in \SDD_n$.
            \end{enumerate}
    \end{corollary}

Corollary~\ref{prop:fwsdd} is proved in Appendix~\ref{app:B} and presents another proof for the equivalence that $\SDD_n = \mathcal{FW}^n_2$. This equivalence was originally proved in~\cite{boman2005factor} which relies on expressing a diagonally dominant matrix $A$ as a sum of rank-1 matrices.
    The alternative description of $\mathcal{FW}^n_{\alpha,2}$ in Theorem~\ref{theo:blockfw} allows for deducing a few useful properties of block factor-width-two matrices.

    \begin{corollary}\label{prop:invariant}
    Given a partition $\alpha = \{k_1, \ldots, k_p\}$ with $\sum_{i = 1}^p k_i = n$, we have the following statements:
            \begin{enumerate}
                \item $A \in \mathcal{FW}^n_{\alpha, 2}$ if and only if $DAD^\tr \in \mathcal{FW}^n_{\alpha,2}$ for any invertible block-diagonal matrix $D = \text{diag}(D_1, \ldots, D_p)$, where $D_i \in \mathbb{R}^{k_i}, i = 1, \ldots, p$.
                \item For any $X \in \mathbb{S}^n$, there exist $A, B \in \mathcal{FW}^n_{\alpha,2}$ such that $X = A - B$.
                \item $\mathcal{FW}^n_{\alpha,2}$ is invariant with respect to block $\alpha$-permutation, \emph{i.e.}, $A \in \mathcal{FW}^n_{\alpha, 2}$ if and only if $P_{\alpha}AP_{\alpha}^\tr \in \mathcal{FW}^n_{\alpha,2}$.
            \end{enumerate}
    \end{corollary}

\begin{proof}
    \emph{Statement 1:} Suppose $A \in \mathcal{FW}^n_{\alpha,2}$. By Theorem~\ref{theo:blockfw}, there exist  $Z_{ij} \in \mathbb{S}^{k_i}_+, i, j = 1, \ldots, p, i \neq j$, such that~\eqref{eq:blockfw2_s1} and~\eqref{eq:blockfw2_s2} hold. Then we have
    $$
        D_iA_{ii}D_i^\tr \succeq \sum_{j =1, j \neq i}^p D_iZ_{ij}D_i^\tr, \quad \forall i = 1, \ldots, p
        $$
        and
        $$
        \begin{aligned}
       &\begin{bmatrix} D_{i} &  \\
         & D_{j}\end{bmatrix} \begin{bmatrix} Z_{ij} & A_{ij} \\
        \star & Z_{ji}\end{bmatrix} \begin{bmatrix} D_{i} &  \\
         & D_{j}\end{bmatrix}^\tr\\
         = &\begin{bmatrix} D_{i}Z_{ij}D_{i}^\tr & D_{i}A_{ij}D_{j}^\tr \\
        \star & D_{j}Z_{ji}D_{j}^\tr\end{bmatrix} \succeq 0, \quad \forall \,1 \leq i < j \leq p.
        \end{aligned}
  $$
  Thus, setting $\hat{Z}_{ij} = D_{i}Z_{ij}D_{i}^\tr $ proves $DAD^\tr \in \mathcal{FW}^n_{\alpha,2}$. The converse  follows by observing that
  $$
    DAD^\tr \in \mathcal{FW}^n_{\alpha,2} \Rightarrow A = D^{-1}DAD^\tr (D^{-1})^\tr \in \mathcal{FW}^n_{\alpha,2}.
  $$

  \emph{Statement 2:} Given $X = [X_{ij}] \in \mathbb{S}^n$ with partition $\alpha$, we can choose $A = X + \lambda I_n$ and $B = \lambda I_n$, which satisfies
  $
    X = A - B, \, \forall \lambda \in \mathbb{R}.
  $
  Since $B$ is diagonal and the off-diagonal elements are zero, the constraints~\eqref{eq:blockfw2_s1} and~\eqref{eq:blockfw2_s2} can be naturally satisfied $\forall \lambda >0$, and thus $B \in \mathcal{FW}^n_{\alpha,2}$. Meanwhile, the diagonal elements of $A$ can be chosen large enough by considering a large $\lambda >0$, such that the diagonal elements of $Z_{ij}$ in~\eqref{eq:blockfw2_s1} are large enough to satisfy~\eqref{eq:blockfw2_s2}. From Theorem~\ref{theo:blockfw}, it is now clear that there exists a $\lambda >0 $ such that $A, B \in \mathcal{FW}^n_{\alpha,2}$.

  \emph{Statement 3:} Follows directly from the fact that~\eqref{eq:blockfw2_s1} and~\eqref{eq:blockfw2_s2} are independent of block $\alpha$-permutation. 
\end{proof}

Statement 2 of Corollary~\ref{prop:invariant} can be used to provide additional results for the difference of convex (DC) decomposition of nonconvex polynomials that was initially proposed in~\cite{ahmadi2018dc}. We will not discuss further this application, but mention that it remains to establish how $\mathcal{FW}_{\alpha,2}^n$ matrices can be used in this context. The block invariant property in Statement 3 of Corollary~\ref{prop:invariant} has been used in the proof of Theorem~\ref{theo:inclusion}. 
Finally, we note that the cone $\mathcal{FW}^n_{\alpha,2}$ is not invariant with respect to the normal permutation, unless $\alpha = \{1,1, \ldots, 1\}$. In other words,  given a nontrivial partition $\alpha$ and $A \in \mathcal{FW}^n_{\alpha,2}$, we may have
$
    PAP^\tr \notin \mathcal{FW}^n_{\alpha,2},
$
where $P$ is a standard $n \times n$ permutation matrix.

\section{Approximation quality of block factor-width-two matrices} \label{section:approximation}

For any partition $\alpha$, we have $\mathcal{FW}_{\alpha,2}^n \subseteq \mathbb{S}^n_+ \subseteq (\mathcal{FW}_{\alpha,2}^n)^*$, that is to say $\mathcal{FW}_{\alpha,2}^n$ and its dual $(\mathcal{FW}_{\alpha,2}^n)^*$ serve as inner and outer approximations of the PSD cone $\mathbb{S}^n_+$, respectively. In this section, we aim to analyze how well $\mathcal{FW}_{\alpha,2}^n$ and $(\mathcal{FW}_{\alpha,2}^n)^*$ approximate the PSD cone. We focus on two cases: 1) general dense matrices, where we present upper and lower bounds on the distance between $\mathcal{FW}_{\alpha,2}^n$ (or $(\mathcal{FW}_{\alpha,2}^n)^*$) and $\mathbb{S}^n_+$ after some normalization; 2) a class of sparse block PSD matrices, for which there is no approximation error to use our notion of block factor-width-two matrices.

\subsection{Upper and lower bounds} \label{Section:bounds}
Our results in this section are motivated by~\cite{blekherman2020sparse}, where the authors quantified the approximation quality for the PSD cone using the dual cone of factor-width-$k$ matrices $(\mathcal{FW}^n_k)^*$.
In our context, there are two cases:
\begin{itemize}
    \item For the case of $\mathcal{FW}_{\alpha,2}^n \subseteq \mathbf{S}^n_+$, we consider the matrix in $\mathbf{S}^n_+$ that is farthest from $\mathcal{FW}_{\alpha,2}^n$.
    \item For the case of $\mathbf{S}^n_+ \subseteq (\mathcal{FW}_{\alpha,2}^n)^*$, we consider the matrix in $(\mathcal{FW}_{\alpha,2}^n)^*$ that is farthest from the PSD cone $\mathbf{S}^n_+$.
\end{itemize}
The distance between a matrix $M$ and a set $\mathcal{D}$ (where $\mathcal{D} = \mathbb{S}^n_+$ or $\mathcal{FW}_{\alpha,2}^n$) is measured as
$
    \text{dist}(M,\mathcal{D}) := \inf_{N \in \mathcal{D}} \|M - N\|_F,
$
where $\|\cdot\|_F$ denotes the Frobenius norm. Similar to~\cite{blekherman2020sparse}, we consider the following normalized Frobenius distances:
\begin{itemize}
        \item The largest distance between a unit-norm matrix $M$ in $\mathbb{S}^n_+$ and the cone $\mathcal{FW}_{\alpha,2}^n$:
    $$
        \text{dist}(\mathbb{S}^n_+,\mathcal{FW}_{\alpha,2}^n) := \sup_{M \in \mathbb{S}^n_+, \|M\|_F = 1}  \text{dist}(M,\mathcal{FW}_{\alpha,2}^n).
    $$
    \item The largest distance between a unit-norm matrix $M$ in $(\mathcal{FW}_{\alpha,2}^n)^*$ and the PSD cone $\mathbb{S}^n_+$:
    $$
     \!\!\!   \text{dist}((\mathcal{FW}_{\alpha,2}^n)^*,\mathbb{S}^n_+) := \sup_{M \in (\mathcal{FW}_{\alpha,2}^n)^*, \|M\|_F = 1}  \text{dist}(M,\mathbb{S}^n_+).
    $$
\end{itemize}
We aim to characterize the bounds for $\text{dist}(\mathbb{S}^n_+,\mathcal{FW}_{\alpha,2}^n)$ and $\text{dist}((\mathcal{FW}_{\alpha,2}^n)^*,\mathbb{S}^n_+)$.

\subsubsection{Upper bound} We first show that the distance between $(\mathcal{FW}_{\alpha,2}^n)^*$ (or $\mathcal{FW}_{\alpha,2}^n$) and the PSD cone is at most $\frac{p-2}{p}$, where $p$ is the number of partitions.
\begin{proposition} \label{prop:upperbound}
        For any partition $\alpha = \{k_1,k_2,\ldots,k_p\}$, we have
    $$
        \text{dist}(\mathbb{S}^n_+,\mathcal{FW}_{\alpha,2}^n) \leq\frac{p-2}{p}, \quad  \text{dist}((\mathcal{FW}_{\alpha,2}^n)^*,\mathbb{S}^n_+) \leq \frac{p-2}{p}.
    $$
\end{proposition}
The proof is provided in Appendix\ref{app:C}. We note that the block factor-width-two approximation becomes exact when $p = 2$. As expected, this upper bound roughly goes from 0 to 1 as the number of partitions $p$ goes from 2 to $n$. We note that the normalized distance between $(\mathcal{FW}^n_k)^*$ and $\mathbb{S}^n_+$ has an upper bound as~\cite{blekherman2020sparse}
\begin{equation} \label{eq:upperfwk}
    \text{dist}((\mathcal{FW}_{k}^n)^*,\mathbb{S}^n_+) \leq \frac{n-k}{n+k-2}.
\end{equation}
From the upper bounds in Proposition~\ref{prop:upperbound} and~\eqref{eq:upperfwk}, it is explicitly shown that reducing the number of partitions $p$ or increasing the factor-width $k$ can improve the approximation quality for the PSD cone. We note that reducing the number of partitions is more efficient in numerical computation, since the decomposition basis in~\eqref{eq:BlkFW} is reduced, while the decomposition basis for $(\mathcal{FW}_{k}^n)^*$ is a combinatorial number ${n \choose k}$.

\subsubsection{Lower bound} We next provide a lower bound on $\text{dist}((\mathcal{FW}_{\alpha,2}^n)^*,\mathbb{S}^n_+) $ for a class of block matrices with homogeneous partition $\alpha$.
\begin{proposition} \label{prop:lowerbound}
    Given a homogeneous partition $\alpha$, we have
    $$
        \text{dist}((\mathcal{FW}_{\alpha,2}^n)^*,\mathbb{S}^n_+) \geq \frac{1}{\sqrt{\frac{4n}{p^2} - \frac{4}{p}+1}}\frac{p-2}{p}.
    $$
\end{proposition}

The proof is adapted from~\cite{blekherman2020sparse} and is provided in Appendix\ref{app:D} for completeness. For homogeneous partitions, the upper bound in Proposition~\ref{prop:upperbound} matches well with the lower bound in Proposition~\ref{prop:lowerbound}.
Given a homogeneous partition $\alpha$, we have $(\mathcal{FW}_{2 k}^n)^* \subseteq (\mathcal{FW}_{\alpha, 2}^n)^*$ with $p = n/k$,  and 
\begin{equation*}
        \text{dist}((\mathcal{FW}_{2 k}^n)^*,\mathbb{S}^n_+) \leq \frac{p-2}{p+2-2/k},
\end{equation*}
which shows that the distances $\text{dist}((\mathcal{FW}_{2 k}^n)^*,\mathbb{S}^n_+)$, $\text{dist}((\mathcal{FW}_{\alpha,2}^n)^*,\mathbb{S}^n_+)$ are growing increasingly close with growing $k$. Also, trivially $\text{dist}((\mathcal{FW}_{\alpha,2}^n)^*,(\mathcal{FW}_{2 k}^n)^*)\le \frac{p-2}{p}$.
Estimating a tighter bound, however, is a challenging task.
\begin{remark}
    Propositions~\ref{prop:upperbound} and~\ref{prop:lowerbound} present an explicit quantification of the approximation quality when varying the number of block partitions. Intuitively, the block sizes in a partition will affect the approximation quality as well; however, estimating approximation bounds in terms of different block sizes is challenging. We note that choosing the block sizes may be problem dependent. For example, in networked control applications, each block size may correspond to the dimension of each subsystem~\cite{zheng2017scalable}. In Section~\ref{section:sosmatrix}, we show that there exists a natural block partition for a class of polynomial optimization problems.
\end{remark}

\subsection{Sparse block factor-width-two matrices}

Here, we identify a class of sparse PSD matrices that always belongs to $\mathcal{FW}^n_{\alpha,2}$, which means there is no approximation error for this class of PSD matrices.
First, Theorem~\ref{theo:blockfw} allows us to deal with the sparsity of the matrix $A$ in an efficient way, as shown in the following result.
\begin{corollary}
    Given $A \in \mathcal{FW}^n_{\alpha,2}$, let $\mathcal{E} = \{(i,j) \mid \|A_{ij}\|_2 \neq 0\}$, then we have
    $$
        A = \sum_{(i,j) \in \mathcal{E}, i < j} (E^{\alpha}_{ij})^\tr X_{ij} E_{ij}^{\alpha}.
    $$
    \end{corollary}
    \begin{proof}
        This directly follows from the proof of Theorem~\ref{theo:blockfw}. Indeed, if $A_{ij} = 0$, the corresponding $Z_{ij}$ and $Z_{ji}$ can be set to zero in~\eqref{eq:blockfw2_s1} and~\eqref{eq:blockfw2_s2}. Then, the component $E^{\alpha}_{ij}$ corresponding to $A_{ij} = 0$ can be set to zero as the variable $X_{ij}$ is block-diagonal and can be incorporated into other components.
    \end{proof}

\begin{proposition}\label{prop:PSDfw}
        Given a partition $\alpha = \{k_1, k_2, \ldots, k_p\}$ and a chordal graph $\mathcal{G}(\mathcal{V}, \mathcal{E})$ with $\mathcal{V} = \{1, \ldots, p\}$, if the largest maximal clique size is two, then we have
 $
            \mathbb{S}^n_{\alpha,+}(\mathcal{E},0) \subset \mathcal{FW}^n_{\alpha,2}. 
            $
\end{proposition}
\begin{proof}
    According to Theorem~\ref{prop:chordal}, $\forall Z \in \mathbb{S}^n_{\alpha,+}(\mathcal{E},0)$, we have
        $$ 
            Z = \sum_{i=1}^g (E^{\alpha}_{\mathcal{C}_i})^\tr Z_i (E^{\alpha}_{\mathcal{C}_i}),
        $$ 
        where  $Z_i \in \mathbb{S}^{|\mathcal{C}_i|}_+, i = 1, \ldots, g$. When the largest maximal clique size is two, then the basis matrices above belong to $E^{\alpha}_{ij}$. Thus, $Z \in \mathcal{FW}^n_{\alpha,2}$. 
\end{proof}

Using the distance notion in Section~\ref{Section:bounds}, we have that
$
    \text{dist}(\mathbb{S}^n_{\alpha,+}(\mathcal{E},0),\mathcal{FW}_{\alpha,2}^n) =0.
$
We conclude this section with Fig.~\ref{FIG:4} that illustrates two sparsity patterns where Proposition~\ref{prop:PSDfw} is applicable. Note that Proposition~\ref{prop:PSDfw} works for any partition, and this allows us to conclude that any second-order constraint (see Fig.~\ref{FIG:4}(a)) can be represented by a factor-width-two constraint $\mathcal{FW}^n_2$.

\begin{figure}
	\centering
	\subfigure[]{
		\includegraphics[scale=.5]{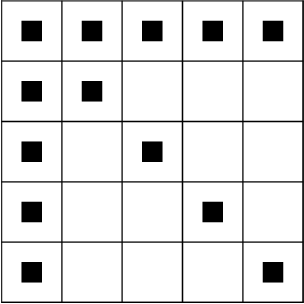}
	}
	\hspace{5mm}
	\subfigure[]{
		\includegraphics[scale=.5]{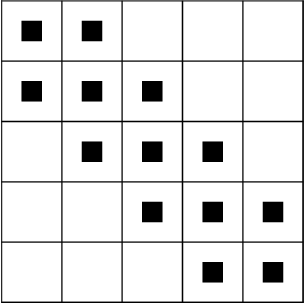}
	}
	\caption{Two sparsity patterns $\mathcal{E}$ where we have $\mathbb{S}^n_{\alpha,+}(\mathcal{E},0) \subset \mathcal{FW}^n_{\alpha,2}$ 
	for any partition $\alpha$: (a) an arrow pattern, (b) a tridiagonal pattern. Each black square represents a block entry of compatible dimension. }
	\label{FIG:4}
\end{figure}

\section{Applications in semidefinite and sum-of-squares optimization} \label{section:Application}

In this section, we discuss some applications of block factor-width-two matrices in semidefinite and sum-of-squares optimization. The optimization problems below are all standard and have a wide range of applications; see~\cite{vandenberghe1996semidefinite,blekherman2012semidefinite,bernard2009moments} for excellent surveys.

\subsection{Inner approximation of SDPs}
Given $b \in  \mathbb{R}^m$, $C \in \mathbb{S}^n$, and matrices $A_1, \ldots, A_m \in \mathbb{S}^n$, the standard primal form of a semidefinite program (SDP) is
\begin{equation} \label{eq:SDPprimal}
\begin{aligned}
    \min_{X} \quad & \langle C, X \rangle \\
    \text{subject to} \quad & \langle A_i, X \rangle  = b_i, i = 1, \ldots, m, \\
    &  X \in  \mathbb{S}^n_+,
\end{aligned}
\end{equation}
where $\langle \cdot, \cdot \rangle$ denotes the trace inner product in $\mathbb{S}^n$. Using  Theorem~\ref{theo:inclusion}, we can build an inner approximation of~\eqref{eq:SDPprimal} by replacing $\mathbb{S}^n_+$ with a block factor-width-two cone $\mathcal{FW}^n_{\alpha,2}$,
\begin{equation} \label{eq:SDPfw}
\begin{aligned}
    \min_{X} \quad & \langle C, X \rangle \\
    \text{subject to} \quad & \langle A_i, X \rangle  = b_i, i = 1, \ldots, m, \\
    &  X \in  \mathcal{FW}^n_{\alpha,2}.
\end{aligned}
\end{equation}
We call Problem~\eqref{eq:SDPfw} as a block factor-width-two cone program. Note that any $X$ that is a feasible solution to~\eqref{eq:SDPfw} is also feasible to~\eqref{eq:SDPprimal}.
This justifies the notion of \emph{inner approximation}. One can also replace $X \in \mathbb{S}^n_+$ with $X \in (\mathcal{FW}^n_{\alpha,2})^*$, leading to an outer approximation of~\eqref{eq:SDPprimal}, and the solution may not be feasible to~\eqref{eq:SDPprimal}. In the following, we are mainly interested in the inner approximation~\eqref{eq:SDPfw}.

Denoting the optimal cost\footnote{We denote optimal cost is infinity if the problem is infeasible.} of~\eqref{eq:SDPprimal} and~\eqref{eq:SDPfw} are $J^*$ and $J^{\alpha}$, respectively, it is easy to see that~\eqref{eq:SDPfw} returns an upper bound of~\eqref{eq:SDPprimal}, \emph{i.e.},
$$
    J^{\alpha} \geq J^*.
$$
When $\alpha = \{1, 1, \ldots, 1\}$, Problem~\eqref{eq:SDPfw} can be reformulated into an SOCP, as extensively used in~\cite{ahmadi2019dsos}. In this case, however, the gap between $J^{\alpha}$ and $J^*$ might be very large. As shown in Theorem~\ref{theo:inclusion}, we can create a coarser partition $\alpha \sqsubseteq \beta$ that improves the solution accuracy $J^{\alpha} \geq J^{\beta} \geq J^*$.

By the definition of $ \mathcal{FW}^n_{\alpha,2}$, Problem~\eqref{eq:SDPfw} can be equivalently rewritten into the standard SDP form
\begin{align}
    \min_{X_{jl}} \quad & \sum_{j=1}^{p-1} \sum_{l = j+1}^{p}\left\langle E^{\alpha}_{jl}C(E^{\alpha}_{jl})^\tr, X_{jl} \right\rangle  \nonumber \\
    \text{subject to} \quad & \sum_{j=1}^{p-1} \sum_{l = j+1}^{p} \left\langle E^{\alpha}_{jl}A_i(E^{\alpha}_{jl})^\tr, X_{jl} \right\rangle  = b_i, i = 1, \ldots, m, \nonumber \\
    &  X_{jl} \in  \mathbb{S}^{k_j + k_l}_{+},  1 \leq j < l \leq p, \label{eq:SDPfw_s1}
\end{align}
which is amenable for a straightforward implementation in standard SDP solvers such as SeDuMi~\cite{sturm1999using} and MOSEK~\cite{andersen2000mosek}. This program has the same number of equality constraints as~\eqref{eq:SDPprimal}, but the dimensions of PSD constraints have been reduced. This reformulation~\eqref{eq:SDPfw_s1} often offers computational speed improvements as demonstrated in our numerical experiments.

\subsection{Inner approximation of SOS optimization}
The notion of block factor-width-two matrices can also be applied in SOS optimization. A real coefficient polynomial $p(x)$ is a sum-of-squares (SOS) polynomial if it can be written as a finite sum of squared polynomials, \emph{i.e.},
\begin{equation}\label{eq:SOSdefinition}
    p(x) = \sum_{i=1}^m q^2_i(x)
\end{equation}
for some polynomial $q_i$.
We denote the set of SOS polynomials in $n$ variables and of degree $2d$ as $\SOS_{n,2d}$. It is clear that $p(x) \in \SOS_{n,2d}$ implies $p(x) \geq 0, \forall x \in \mathbb{R}^n$. We refer the interested reader to~\cite{blekherman2012semidefinite,bernard2009moments} for extensive applications of SOS polynomials.

It is well-known that $p(x) \in \SOS_{n,2d}$ if and only if there exists a PSD matrix $Q \succeq 0$ such that~\cite{blekherman2012semidefinite,bernard2009moments}
\begin{equation} \label{eq:SOSpsd}
    p(x) = v_d(t)^\tr Q v_d(x),
\end{equation}
where
\begin{equation} \label{eq:monomialbasis}
v_d(x) = [ 1,x_1,x_2,\ldots,x_n,x_1^2,x_1x_2,\ldots,x_n^d ]^\tr
\end{equation}
is the vector of monomials in $x$ of degree $d$ or less. One fundamental computational challenge in optimization over $\SOS_{n, 2d}$ is that the parameterization~\eqref{eq:SOSpsd} requires a ${n+d \choose d} \times {n+d \choose d}$ PSD matrix. This may be prohibitive even for moderate sizes $n$ and $d$. Numerous efforts have been devoted to improve the scalability of SOS optimization~\cite{weisser2018sparse,gatermann2004symmetry,waki2006sums,zheng2019sparse,zheng2018fast}. A recent notion is so-called \emph{scaled diagonally dominant sum-of-squares (\SDSOS)}~\cite{ahmadi2019dsos} that is based on $\mathcal{FW}^n_2$.

Motivated by~\cite{ahmadi2019dsos}, we define a block version of \SDSOS~based on $\mathcal{FW}^n_{\alpha,2}$, which offers an improved inner approximation of $\SOS_{n,2d}$.
\begin{definition}
    Given a partition $\alpha = \{k_1, k_2, \ldots, k_g\}$ with $\sum_{i=1}^g k_i = {n+d \choose d}$, we call a polynomial $p(x)$ in $n$ variables and of degree $2d$ as $\alpha$-$\SDSOS_{n,2d}$, if it can be represented as
    \begin{equation} \label{eq:sdsos}
        p(x) = \sum_{1 \leq i < j \leq g}\left(\sum_{t=1}^{k_i + k_j} \left(f_{ij,t}^\tr m_{ij}(x)\right)^2\right),
    \end{equation}
    where $f_{ij,t} \in \mathbb{R}^{k_i + k_j}$ and $m_{ij}(x)$ is a subvector of length $k_i+k_j$ of the monomial basis vector~\eqref{eq:monomialbasis}.
\end{definition}

\begin{remark}
Definition~\eqref{eq:sdsos} is more structured compared to the SOS definition~\eqref{eq:SOSdefinition}, thus it is clear that $\alpha$-$\SDSOS_{n,2d} \subseteq \SOS_{n,2d}$. We note that for a trivial partition $\alpha = \{1,1,\ldots, 1\}$, our notion of  $\alpha$-$\SDSOS_{n,2d}$ is reduced to the normal \SDSOS \; in~\cite{ahmadi2019dsos}. For a partition $\alpha = \{k_1, k_2\}$ with $k_1 + k_2 = {n+d \choose d}$, it is not difficult to see
$\alpha$-$\SDSOS_{n,2d} = \SOS_{n,2d}$, by invoking a Cholesky factorization of the PSD matrix $Q$ in~\eqref{eq:SOSpsd}.

\end{remark}

The following theorem connects our $\alpha$-$\SDSOS_{n,2d}$ polynomials with block factor-width-two matrices.

\begin{theorem} \label{theo:sdsos}
    A polynomial $p(x) \in \alpha$-$\SDSOS_{n,2d}$ if and only if it admits a representation as $p(x) = v_d^\tr(x)Qv_d(x)$, where $Q \in \mathcal{FW}^N_{\alpha, 2}$, where $N = {n+d \choose d}$.
\end{theorem}

\begin{proof}
    $\Leftarrow:$ If $p(x)$ admits a presentation
    $$
      p(x) = v_d^\tr(x)Qv_d(x),
    $$
    where $Q \in \mathcal{FW}^N_{\alpha, 2}$, then we have
    $$
        Q = \sum_{1 \leq i < j \leq g} (E^{\alpha}_{ij})^{\tr} Q_{ij}E^{\alpha}_{ij},
        $$
    for some $Q_{ij} \in \mathbb{S}^{k_i + k_j}_+$. Using a Cholesky factorization $Q_{ij} = F_{ij} F^\tr_{ij}$, we have
    $$
    \begin{aligned}
        p(x) &= v^\tr_d(x) \left( \sum_{1 \leq i < j \leq g} (E^{\alpha}_{ij})^{\tr} Q_{ij}E^{\alpha}_{ij},
         \right)v_d(x) \\
         & = \sum_{1 \leq i < j \leq g} \left( \left(E^{\alpha}_{ij}v_d(x)\right)^\tr Q_{ij}\left(E^{\alpha}_{ij}v_d(x)\right)\right) \\
         &=\sum_{1 \leq i < j \leq g} \left( \left(F_{ij}^\tr E^{\alpha}_{ij}v_d(x)\right)^\tr \left(F_{ij}^\tr E^{\alpha}_{ij}v_d(x)\right)\right).
    \end{aligned}
        $$
        By denoting $m_{ij}(x) = E^{\alpha}_{ij}v_d(x)$, we arrive at the conclusion that $p(x) \in \alpha$-$\SDSOS_{n,2d}$.

        $\Rightarrow:$ Now suppose $p(x) \in \alpha$-$\SDSOS_{n,2d}$. By definition, we have
        $$
              p(x) = \sum_{1 \leq i < j \leq g}\left(\sum_{t=1}^{k_i + k_j} \left(f_{ij,t}^\tr m_{ij}(x)\right)^2\right),
        $$
        We can construct
        $$
            F_{ij} = \begin{bmatrix} f_{ij, 1} & f_{ij, 2} & \ldots & f_{ij, k_i + k_j} \end{bmatrix},
        $$
        leading to a $Q_{ij} = F_{ij}F_{ij}^\tr \in \mathbb{S}^{k_i + k_j}_+$. And $Q \in \mathcal{FW}^N_{\alpha,2}$ is constructed accrodingly.
\end{proof}

Similar to Theorem~\ref{theo:inclusion}, we can build a hierarchy of inner approximations of $\SOS_{n,2d}$. This is stated in the following corollary.
\begin{corollary}\label{coro:sdsosinclusion}
        For polynomials $p(x)$ in $n$ variables and of degree $2d$, we consider three partitions $\alpha = \{k_1, \ldots, k_g\}, \beta = \{l_1, \ldots, l_h\}$ and $\gamma = \{N_1,N_2\}$, where $\sum_{i=1}^g k_i = \sum_{i=1}^h l_i =N_1 + N_2 = N = {n+d \choose d}$ and $\alpha  \sqsubseteq \beta$. Then, we have the following inclusion:
    $$
    \begin{aligned}
        \SDSOS_{n,2d} = &\mathbf{1}\text{-}\SDSOS_{n,2d} \subseteq  \alpha\text{-}\SDSOS_{n,2d} \\ &\subseteq  \beta\text{-}\SDSOS_{n,2d} \subseteq  \gamma\text{-}\SDSOS_{n,2d} = \SOS_{n, 2d},
    \end{aligned}
        $$
    where $\mathbf{1} = \{1, \ldots, 1\}$ denotes the trivial partition.
\end{corollary}

\begin{proof}
    The proof directly follows by combining Theorems~\ref{theo:inclusion} and~\ref{theo:sdsos}.
\end{proof}

We also have the following corollary corresponding to Corollary~\ref{prop:invariant}.
\begin{corollary}
    For any polynomial $p(x)$ in $n$ variables and of degree $2d$, there exist $q(x), h(x) \in \alpha$-$\SDSOS_{n,2d}$ such that
    $$
        p(x) = q(x) - h(x).
    $$
\end{corollary}

Given {$p_0(x),\,\ldots,\,p_t(x)$} in $n$ variables and of degree $2d$, $w \in \mathbb{R}^t$, and a partition $\alpha = \{k_1, k_2, \ldots, k_g\}$, we define the following problem,
\begin{equation}
\label{E:generalSOS}
    \begin{aligned}
        \min_{u}\quad & w^\tr u \\[-1ex]
        \text{subject to} \quad  & p_0(x) + \sum_{i=1}^t u_ip_i(x) \geq 0, \forall x \in \mathbb{R}^n,
    \end{aligned}
\end{equation}
where $u \in \mathbb{R}^t$ is the decision variable. Replacing the non-negativity in~\eqref{E:generalSOS} with a constraint of $\SOS_{n,2d}$, $\alpha\text{-}\SDSOS_{n,2d}$, or $\SDSOS_{n,2d}$ leads to an SDP, block factor-width cone program, or SOCP, respectively. The constraint $\SOS_{n,2d}$ provides the best solution quality but requires the most computation; the constraint $\SDSOS_{n,2d}$ demands the least computational resources 
but it may be too restrictive and leads to an infeasible restriction. Instead, the constraint $\alpha\text{-}\SDSOS_{n,2d}$ can balance the computational speed and solution quality by varying the partition $\alpha$. This feature is confirmed by our numerical experiments in Section~\ref{section:Experiments}.

\subsection{Inner approximation of PSD polynomial matrices} \label{section:sosmatrix}

The partition $\alpha$ offers flexibility to approximate both $\mathbb{S}^n_+$ and $\SOS_{n,2d}$, as demonstrated in Theorem~\ref{theo:inclusion} and Corollary~\ref{coro:sdsosinclusion}. Choosing an appropriate partition might be problem dependent. Here, we show that for approximating PSD polynomial matrices, there exists a natural partition.

An $r \times r$ polynomial matrix $P(x)$ in $n$ variables and of degree $2d$ is PSD if
$$
    P(x) = \begin{bmatrix}
    p_{11}(x) & \ldots & p_{1r}(x)\\
    \vdots & \ddots & \vdots\\
    p_{r1}(x) & \ldots & p_{rr}(x)\\
    \end{bmatrix} \succeq 0, \forall x \in \mathbb{R}^n.
$$
Testing whether $P(x)$ is PSD is in general intractable, and a standard technique is the SOS relaxation~\cite{gatermann2004symmetry}:  $P(x)$ is called an SOS matrix if there exists another polynomial matrix of dimension $s \times r$ such that $P(x) = M(x)^\tr M(x)$.
Under this definition, it is not difficult to prove that $P(x)$ is SOS if and only if
$
y^\tr P(x) y$ is SOS in $[x;y]$; see e.g.,\cite[Proposition 2]{zheng2019sparse}.
Then, in principle, the hierarchy of inner approximations in Corollary~\ref{coro:sdsosinclusion} can be used. For this particular application of SOS matrices, we show that there exists a natural partition $\alpha$.

Our idea is to use a standard factor-width-two decomposition on $P(x)$, assuming
$$
    P(x) = \sum_{1 \leq i < j \leq r} E^\tr_{ij} P_{ij}(x)E_{ij}, \quad P_{ij}(x) \succeq 0, \forall x \in \mathbb{R}^n,
$$
where $P_{ij}(x)$ are $2 \times 2$ polynomial matrices. Then, we apply the SOS relaxation for each $P_{ij}(x)$, \emph{i.e.},
\begin{equation}\label{eq:SOSmatrices}
    P(x) = \sum_{1 \leq i < j \leq r} E^\tr_{ij} P_{ij}(x)E_{ij},\quad  P_{ij}(x) \;\; \text{is \SOS}.
\end{equation}
It is easy to see~\eqref{eq:SOSmatrices} is a sufficient condition for $P(x) \succeq 0$. The choice~\eqref{eq:SOSmatrices} indeed implies a special partition $\alpha$ in the Gram matrix representation of SOS matrices.

\begin{proposition} \label{prop:SOSmatrices}
        Given an $r \times r$ polynomial matrix in $n$ variables and of degree $2d$, $P(x)$ admits a decomposition~\eqref{eq:SOSmatrices} if and only if we have
        \begin{equation} \label{eq:SOSmatricesFW}
            P(x) = \left(I_r \otimes v_d(x)\right)^\tr Q \left(I_r \otimes v_d(x)\right),
        \end{equation}
        where $Q \in \mathcal{FW}^{rN}_{\alpha,2}$ with a homogeneous partition
        $$\alpha = \{\underbrace{N,N,\ldots, N}_{r}\}, \quad N = {n+d \choose d}.$$
\end{proposition}
\begin{proof}
    $\Leftarrow: $ Suppose we have~\eqref{eq:SOSmatricesFW}. Then, we have
    $$
        Q = \sum_{1 \leq i < j \leq r} (E^{\alpha}_{ij})^\tr Q_{ij} E^{\alpha}_{ij}, \quad Q_{ij} \in \mathbb{S}^{2N}_{+}.
    $$
    Considering the definitions of $\alpha$ and $E^{\alpha}_{ij}$, we have
    $$
        E^{\alpha}_{ij} = E_{ij} \otimes I_N.
    $$
    Now, we have
    \begin{equation*}
    \begin{aligned}
        P(x) &= \left(I_r \otimes v_d(x)\right)^\tr \left(\sum_{1 \leq i < j \leq r}  (E^{\alpha}_{ij})^\tr Q_{ij} E^{\alpha}_{ij}\right) \left(I_r \otimes v_d(x)\right)\\
        &= \sum_{1\leq i<j\leq r} E_{ij}^\tr \left( (I_2 \otimes v_d(x))^\tr Q_{ij}(I_2 \otimes v_d(x))\right) E_{ij}\\
        & = \sum_{1 \leq i < j \leq r} E_{ij}^\tr P_{ij}(x) E_{ij},
    \end{aligned}
   \end{equation*}
    where
    $
    P_{ij}(x) = (I_2 \otimes v_d(x))^\tr Q_{ij}(I_2 \otimes v_d(x))
    $ is a $2 \times 2$ SOS matrix.

     $\Rightarrow:$ This direction is similar by just reversing the argument above.
\end{proof}

\begin{remark}
    When a polynomial matrix $P(x)$ has a large matrix dimension $r$ and each polynomial entry has small number $n$ and $d$, Proposition~\ref{prop:SOSmatrices} will lead to better computational efficiency, but the approximation quality may be worse. In addition, similar to the case of sparse PSD matrices in Proposition~\ref{prop:chordal}, if an SOS matrix $P(x)$ has chordal sparsity, a similar decomposition is guaranteed under a mild condition; see~\cite{zheng2019sparse,zheng2018decomposition} for details.
\end{remark}

\section{Numerical experiments} \label{section:Experiments}

To demonstrate the numerical properties of block factor-width-two matrices, we first present two explicit counterexamples, and then consider a set of polynomial optimization problems. Finally, the problem of estimating a region of attraction for polynomial dynamical systems is considered.
In our experiments, we used YALMIP~\cite{lofberg2004yalmip} to reformulate the polynomial optimization problems into standard SDPs, and then replaced the PSD cone with the cone $\mathcal{FW}^n_{\alpha,2}$. All SDP instances were solved by MOSEK\cite{andersen2000mosek}\footnote{Code is available at \url{https://github.com/soc-ucsd/SDPfw}.}.

\subsection{Two explicit counterexamples}

{Fig.~\ref{FIG:3}} has already shown that our notion of $\mathcal{FW}^n_{\alpha,2}$ enriches the cone of $\mathcal{FW}^n_{2}$ or $\SDD_n$ for the approximation of the PSD cone.  Here is an explicit example,
$$
A = \begin{bmatrix}
22   & -4  &  -3   & -7  &  14   & 18 \\
    -4&    15  &  -1 &  -13 &   -8   & -9 \\
    -3 &   -1  &  29  &   2  &   4  & -21 \\
    -7  & -13  &   2  &  27  &   4 &   3 \\
    14   & -8  &   4  &   4  &  15  &  12 \\
    18   & -9  & -21   &  3   & 12  &  37
\end{bmatrix},
$$
for which we have $A \in \mathbb{S}^6_+, A \in \mathcal{FW}^6_{\beta,2}$ with $\beta = \{2,2,2\}$ but $A \notin \mathcal{FW}^6_2$. Note that $A$ is not SDD either.

In the context of SOS polynomials, the $\alpha$-\SDSOS~is a strictly better inner approximation for SOS polynomials compared to the standard \SDSOS~in~\cite{ahmadi2019dsos}.  For example, consider the polynomial matrix
\begin{equation*}
        P(x) = 	\begin{bmatrix}
	4 x^2 + 9 y^2  &  x+y                 & x+y \\
	x+y                  & 9 x^2 + 4 y^2  & x+y \\
	x+y                  & x+y                  & x^2 + 25 y^2
	\end{bmatrix}.
\end{equation*}
Computing a certificate for
$$P(x) - \frac{63}{200}I_3 \succeq 0, \forall x, y \in \mathbb{R},$$
using the standard \SDSOS~approach leads to an infeasible SOCP, while the $\alpha$-\SDSOS~in~\eqref{eq:SOSmatricesFW} is feasible when the partition $\alpha$ is chosen as in Proposition~\ref{prop:SOSmatrices}, i.e., $\alpha = \{3,3,3\}$.

\subsection{Polynomial optimization} \label{section:pop}
    Here, we consider the following polynomial optimization problem:
\begin{equation} \label{Eq:ex1}
\begin{aligned}
\min_{\gamma} \quad & \gamma \\
\text{subject to} \quad & p(x) + \gamma \geq 0,\,\,\, \forall x\in \mathbb{R}^n,
\end{aligned}
\end{equation}
where $q(x)$ is a modified Broyden tridiagonal polynomial
\begin{multline*}
q(x) = ((3-2 x_1) x_1-2 x_2+1)^2 \\
+\sum\limits_{i =2}^{n-1} ((3 -2 x_i) x_i - x_{i-1} - 2 x_{i +1} +1)^2  \\
+((3-2 x_n) x_n - x_{n-1} + 1)^2 + \left(\sum\limits_{i = 1}^n x_i\right)^2.
\end{multline*}
We added the last term, so that the structure-exploiting methods in~\cite{waki2006sums, zheng2019sparse} are not suitable here. Upon replacing the non-negativity constraint in~\eqref{Eq:ex1} with an
SOS or $\alpha$-SDSOS condition, this problem can be reformulated as an SDP or block factor-width cone program, respectively. We vary $n$ and obtain SDPs in the standard primal form of different size.

\begin{table}
\caption{Computational results for~\eqref{Eq:ex1} using SOS and $\alpha$-SDSOS relaxations, where $n$ denotes the dimension of $x \in \mathbb{R}^n$, and $\infty$ means MOSEK ran out of memory. 
}\label{tab:pop_time}
	\centering
	\begin{tabular}{rrrrrr}
	\toprule
		\multirow{2}{10pt}{$n$}   & \multirow{2}{30pt}{Full SDP} & \multicolumn{4}{c}{Number of blocks $p$ in partition $\alpha$} \\
		&  & $4$  &  $10$      &  $20$    & $50$        \\
		\midrule
		\multicolumn{6}{l}{{Computational Time (seconds)}} \\
		$10$  &  $2.38$  & $1.43$  & $1.29$  & $1.28$  & $1.49$      \\
		$15$ &  $27.3$ & $23.3$ & $15.6$ & $10.1$ & $5.36$ \\
		$20$ &  $489$ & $252$   &  $98.1$ & $66.8$ & $28.1$ \\
		$25$                    &  $\infty$                    & $1970$            &   $783$ & $571$ & $132$\\
		$30$                    &  $\infty$                    & $\infty$         &  $5680$ & $3710$ & $840$\\
		\multicolumn{6}{l}{{Objective values $\gamma$}} \\
		$10$                      &   $-0.9$                    & $-0.45$ &  $134$ &  $483$ & $2120$ \\
		$15$                      &   $-0.92$                   & $-0.75$ & $80.1$ &  $459$ & $2240$ \\
		$20$                      &   $-0.87$                   & $-0.87$ &$-0.11$ &  $251$ & $1910$ \\
		$25$                      &   $\infty$                  & $-1.07$ &$-0.21$ &  $231$ & $1360$ \\
		$30$                      &   $\infty$                  &$\infty$ &$-0.37$ &  $177$ & $1770$ \\
	\bottomrule
	\end{tabular}
\end{table}

In our simulation, the partition $\alpha$ was chosen as follows: we first fix the number of blocks $p$ and let $k = \lfloor N /p\rfloor$, where $\lfloor \cdot \rfloor$ is the floor operation, and $N$ is the size of the PSD cone. Then the partition is composed of $N - k p$ blocks of size $k+1$ followed by $(k + 1) p - N$ blocks of size $k$. The number of SDP constraints in our block factor-width cone program is $p(p-1)/2$, as shown in~\eqref{eq:SDPfw_s1}. The computational times and the corresponding objective values are listed in Table~\ref{tab:pop_time}. It is noticeable that with a finer partition we obtain faster solutions, but they are conservative in terms of the objective value. When guaranteeing feasibility (or solution quality) is important, a coarser partition may be more appropriate as it gives  improved inner approximations of the PSD cone. Instead, when scalability is of primal concern, a finer partition may be a good choice. We note that for large-scale instances $n \geq 25$, MOSEK ran out of memory on our machine. On the other hand, our strategy of using block factor-width-two matrices can still provide a useful upper bound for~\eqref{Eq:ex1}.

\subsection{Matrix SOS problems} \label{ss:maxtrix_sos}
For this numerical example, we show that there exists a natural partition $\alpha$ in the case of the matrix SOS programs, as discussed in Section~\ref{section:sosmatrix}. In particular, we consider the following problem:
\begin{equation}
\begin{aligned}
    \min_{\gamma} \quad & \gamma \\
\text{subject to }\quad & P(x) + \gamma I \succeq 0 ,~~\forall x\in\mathbb{R}^3,
\end{aligned}\label{prog:sos-matrix}
\end{equation}
where $P(x)$ is an $r\times r$ polynomial matrix with each element being a quartic polynomial in three variables, and the coefficient of each element is randomly generated. As suggested in Proposition~\ref{prop:SOSmatrices}, the natural partition $\alpha$ for this case is
$$
\alpha = \{\underbrace{10, 10, \ldots, 10}_{r}\}
$$
since we have ${3+2 \choose 2}=10$. In our simulation, we vary the dimension of $P(x)$ from 25 to 50. The computational results are depicted in Table~\ref{tab:mat_sos}. Our approximation of	$\alpha$-\SDSOS~offers faster computational solutions with almost the same optimal objective $\gamma$ (which is not distinguishable within four significant figures) compared to the standard SOS technique, while the standard \SDSOS~technique~\cite{ahmadi2019dsos} provides even faster solutions, but their quality is worse.

\begin{table}
\caption{Computational results for the instance~\eqref{prog:sos-matrix} using \SOS, $\alpha$-\SDSOS~ and \SDSOS~relaxations. } \label{tab:mat_sos}
	\centering
	\begin{tabular}{lrrrrrr}
	\toprule
		$r$  &    $25$ &  $30$ &  $35$ &  $40$ &  $45$  &  $50$ \\
\midrule
\multicolumn{7}{l}{Computational time} \\
		\SOS               &  $14.4$& $35.9$& $87.2$& $175.0$& $316.0$& $487.8$ \\
	$\alpha$-\SDSOS  & $10.8$& $16.6$ & $25.3$ & $36.0$ & $57.4$ & $71.4$  \\
	\SDSOS	    &  $1.1$& $1.3$ & $1.6$ & $2.1$ & $2.6$ & $3.3$  \\
		\multicolumn{7}{l}{Objective value $\gamma$} \\
		\SOS               & $266.5$& $316.2$& $460.8$& $562.0$& $746.9$& $919.8$ \\
		$\alpha$-\SDSOS &  $266.5$& $316.2$& $460.8$& $562.0$& $746.9$& $919.8$ \\
		\SDSOS       & $270.3$& $324.8$& $477.7$& $570.9$& $762.2$& $961.7$ \\
				\bottomrule
	\end{tabular}
\end{table}

\subsection{Region of attraction (ROA) estimation}
As our final numerical experiment, we consider a control application: estimating the region of attraction (ROA) for a nonlinear dynamical system $$
    \dot{x} = f(x), \quad x(0) = x_0,
$$
where $x(t) \in \mathbb{R}^n$ is the state vector and $f: \mathbb{R}^n \rightarrow \mathbb{R}^n$ is a polynomial vector in $x$. We assume that the origin $x = 0$ is a locally asymptotically stable equilibrium. The region of attraction is defined as
$$
\mathcal{R} := \left\{x_0 \in \mathbb{R}^n \mid  \text{if } x(0) = x_0,\; \text{then}\; \lim_{t\rightarrow \infty} x(t) = 0\right\}.
$$
It is in general very difficult to compute the exact ROA, and significant research has been devoted to estimate an invariant subset of the ROA~\cite{valmorbida2017region,topcu2008local,tan2004searching,mauroy2016global,henrion2013convex,chesi2011domain}. One key technique is based on computing an invariant sublevel set of a Lyapunov function $V(x)$~\cite{vidyasagar2002nonlinear}, \emph{i.e.,} upon defining $\Omega_\gamma := \{x\in \mathbb{R}^n \mid V(x)\leq \gamma\}$, if we have the following set containment relation
\begin{equation}\label{eq:setcontain}
    \Omega_\gamma \subset \{x \in \mathbb{R}^n \mid \nabla V(x)f(x) < 0\} \cup \{0\},
\end{equation}
then an inner approximation $\Omega_\gamma \subset \mathcal{R}$ is obtained. If a Lyapunov function $V(x)$ is given in~\eqref{eq:setcontain}, the set containment constraint can be certified using SOS optimization~\cite{topcu2008local}. For asymptotically stable equilibrium points, a linearization can be used to compute a Lyapunov function, and the so-called V-s iterations can be used to compute better Lyapunov functions; see, \emph{e.g.},~\cite{topcu2008local,tan2004searching}
for details.

\begin{figure}[t]
	\centering
	\subfigure[Full SDP]{
		\includegraphics[scale=.43]{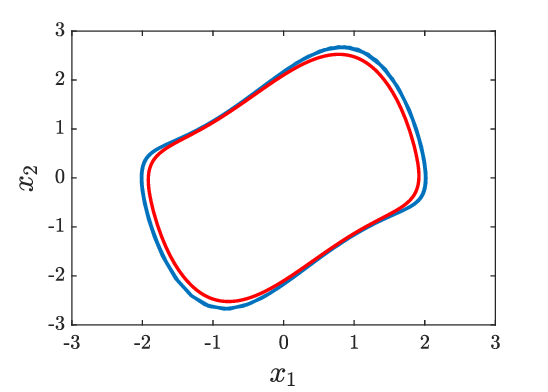}
	}
	\subfigure[Number of blocks: $p=3$]{
		\includegraphics[scale=.43]{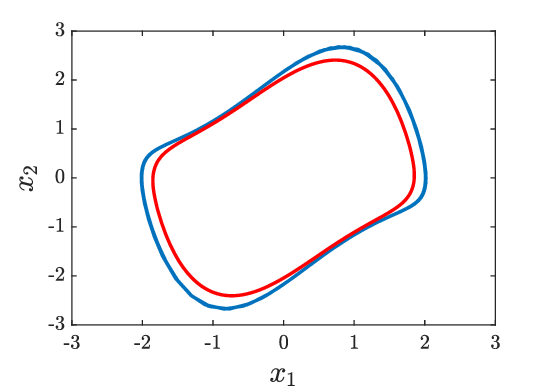}
	}\\
	\subfigure[Number of blocks: $p=5$]{
		\includegraphics[scale=.43]{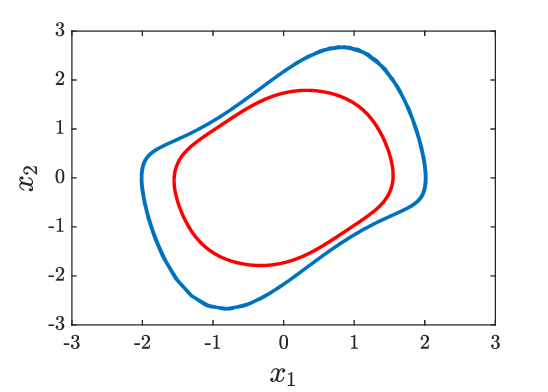}
	}
	\subfigure[SDD approximation]{
		\includegraphics[scale=.43]{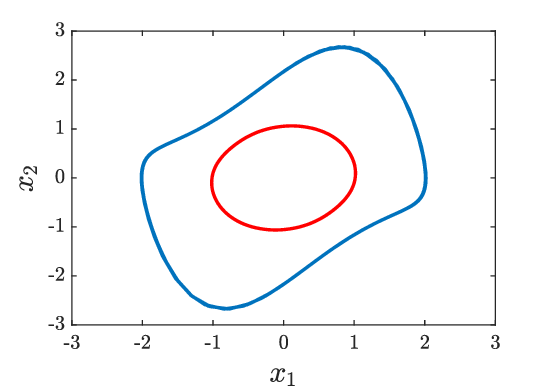}
	}
	\caption{ROA estimation for~\eqref{eq:vdp} (the estimated ROA boundary is highlighted in red) via combining V-s iterations~\cite{topcu2008local,tan2004searching} and block factor-width-two decomposition, where the number of V-s iterations was set to 20: (a) the full SDP was solved at each V-s iteration; (b) the number of block partitions was set to $p=3$; (c) the number of block partitions was set to $p=5$; (d) the SDD approximation was used at each V-s iteration.}
	\label{FIG:5}
\end{figure}

Here, we first consider the classical Van der Pol dynamics in reverse time~\cite{vidyasagar2002nonlinear},
\begin{equation} \label{eq:vdp}
    \begin{aligned}
        \dot{x}_1 &= -x_2, \\
        \dot{x}_2 &=x_1 +(x_1^2-1)x_2,
    \end{aligned}
\end{equation}
which has an unstable limit cycle and an asymptotically stable equilibrium at the origin. We used the V-s iterations\footnote{For the numerical simulation of this subsection, we adapted the code of SOS optimization from~\url{http://users.cms.caltech.edu/~utopcu/saveMaterial/LangleyWorkshop.html}.} to estimate the ROA for~\eqref{eq:vdp}. The results are shown in Fig~\ref{FIG:5}. As expected, it is clear that solving the full SDP returns the best estimate of the ROA, and that increasing the number of block partitions $p$ leads to a more conservative estimation (the SDD approximation gives the most conservative estimation). Since this instance is of relatively small scale, all SDP/block factor-width-two cone program/SOCP instances were solved within one second on our computer.

To demonstrate the scalability issue, we consider the set containment problem~\eqref{eq:setcontain}, which is a subproblem in the ROA estimation. In particular, we check whether the unit ball is contained by the region of $g(x)\leq 0$, \emph{i.e.},
\begin{equation} \label{eq:setcontaint_ex}
    \{x^\tr x < 1\} \subset \{g(x) \leq 0\}
\end{equation}
using SOS optimization. In our simulation, we generated the instance $g(x)$ by setting $g(x) = h^2(x) - 10^3$, where the coefficients of $h(x)$ were randomly generated. The degree of $h(x)$ was set to three, and we varied the dimensions of $x$ in our simulation. The partition $\alpha$ is chosen using the same procedure in Section~\ref{section:pop}. In this case, all the partitions $\alpha$ returned an SOS certificate for the set containment condition~\eqref{eq:setcontaint_ex}. The computational time is listed in Table~\ref{tab:set_time}. As expected, a finer partition $\alpha$ required less time consumption for getting a solution. For the largest instance ($n=13$) we tested, solving the full SDP required more than half an hour to return a solution, while it only took around 6 minutes when using a partition $\alpha$ with $p=50$ blocks, offering $5 \times$ speedup.

\begin{table}
\caption{Computational time (seconds) for~\eqref{eq:setcontaint_ex} using SOS and $\alpha$-SDSOS relaxations, where $n$ denotes the dimension of $x$ and the degree of $g(x)$ is 6 in the simulation. 
}\label{tab:set_time}
	\centering
	{	\begin{tabular}{rrrrrr}
	\toprule
		\multirow{2}{10pt}{$n$}   & \multirow{2}{30pt}{Full SDP} & \multicolumn{4}{c}{Number of blocks $p$ in partition $\alpha$} \\
		&  & $4$  &  $10$      &  $20$    & $50$        \\
		\midrule
		$8$  &  $8.85$  & $8.36$  & $7.82$  & $5.03$  & $4.27$      \\
		$10$ &  $96.5$ & $81.7$ & $61.8$ & $50.9$ & $28.5$ \\
		$12$ &  $802$ & $522$   &  $402$ & $297$ & $231$ \\
		$13$                    &  $1536$                    & $1047$            &   $848$ & $717$ & $378$\\
			\bottomrule
	\end{tabular}
	}
\end{table}

\section{Conclusion} \label{section:Conclusion}

In this paper, we have introduced a new class of block factor-width-two matrices $\mathcal{FW}^n_{\alpha,2}$, and presented a new hierarchy of inner/outer approximations of the cone of PSD matrices.  Our notion of $\mathcal{FW}^n_{\alpha,2}$ is strictly less conservative than the standard $\mathcal{FW}^n_{2}$, and also more scalable than $\mathcal{FW}^n_{k} (k \geq 3)$. Explicit bounds on the approximation quality using block factor-width-two matrices for the PSD cone are provided. Furthermore, we have applied this new class of matrices in both semidefinite and SOS optimization, leading to a new block factor-width cone program and a block extension of \SDSOS~polynomials. As demonstrated in our numerical experiments, via varying the matrix partition, we can balance the trade-off between computational speed and solution quality for large-scale instances. For some applications, the partition comes naturally from the problem formulation, \emph{e.g.}, matrix SOS programs.

Our future work will apply block factor-width-two matrices in relevant control applications that involve SDPs, \emph{e.g.}, network systems where a natural partition may exist according to subsystem dimensions. It would be interesting to further test the performance of our inner and outer approximations for large-scale SDPs in some benchmark problems; e.g.~\cite{Mittelmann,de2009new}. Also, unlike the PSD cone, the cone of $\mathcal{FW}^n_{\alpha,2}$ and its dual $(\mathcal{FW}^n_{\alpha,2})^*$ are not self-dual. It will be interesting to investigate non-symmetric interior-point methods~\cite{skajaa2015homogeneous,nesterov2012towards} to exploit the inherent decomposition structures of $\mathcal{FW}^n_{\alpha,2}$ or $(\mathcal{FW}^n_{\alpha,2})^*$. This may lead to efficient algorithms for solving the block factor-width-two cone program without reformulating it into standard SDPs where additional variables are required. Finally, it is interesting to exploit any underlying sparsity in SDPs before applying the factor-width approximation, e.g., combining chordal decomposition with factor-width approximation, which is promising to improve both numerical efficiency and approximation quality.

{

\appendices

\section*{Appendix}
\addcontentsline{toc}{section}{Appendix}
\renewcommand{\thesubsection}{\Alph{subsection}}

This appendix completes the proof of Corollary~\ref{prop:fwsdd}, Proposition~\ref{prop:upperbound} and Proposition~\ref{prop:lowerbound} in the main text.

\subsection{Proof of Corollary~\ref{prop:fwsdd}}

\label{app:B}

For completeness, we first recall some definitions for the Perron Frobenius theorem; see~\cite{berman1994nonnegative} for more details.

\begin{definition}[Non-negative matrices]
    A matrix is called nonnegative if all its entries are non-negative.
\end{definition}

\begin{definition}[Irreducible  matrices]
    A matrix $A$ is called reducible if it can be conjugated into block upper triangular form by a permutation matrix $P$, \emph{i.e.}
    $$
        PAP^{-1} = \begin{bmatrix}
            B & C\\ 0 & D
        \end{bmatrix},
    $$
    where $B$ and $D$ are non-zero square matrices. A matrix is irreducible if it is not reducible.
\end{definition}

The celebrated Perron Frobenius theorem is as follows.
\begin{lemma}
    Given an irreducible nonnegative matrix $A \in \mathbb{R}^{n \times n}$, we have
    \begin{itemize}
        \item Its spectral radius $\rho(A)$ is a simple eigenvalue of $A$.
        \item The corresponding left and right eigen-vectors can be chosen to be positive element-wise.
        {\item The only element-wise positive eigen-vectors are those corresponding to $\rho(A)$.}
    \end{itemize}
\end{lemma}

We will also need the following lemma 
\begin{lemma}\label{lem:sdd}
Consider a matrix $M = [m_{ij}] \in \mathbb{R}^{n \times n}$ with
\begin{equation}\label{eq:M}
        m_{ij} = \begin{cases}
             -\sum_{j=1, j \neq i}^n z_{ij}, & i = j,\\
            z_{ij},  & i \neq j.
        \end{cases}
\end{equation}
with $z_{i j} \ge 0$. If $M$ has a symmetric non-zero pattern, then there exist positive scalars $d_i > 0$ such that
    \begin{equation} \label{eq:SDD_s0}
        |m_{ii}|d_i^2 \geq \sum_{j=1, j \neq i}^n \left(m_{ji}d_j^2\right), \quad \forall i = 1, \ldots, n.
    \end{equation}
\end{lemma}
\begin{proof}
\emph{Step 1:} The matrix $M$ is either irreducible or $M$ can be conjugated into a block-diagonal matrix
$$
    PMP^{-1} = \begin{bmatrix} M_1 & & \\ & \ddots & \\ & & M_t
    \end{bmatrix},
$$
where $P$ is a permutation matrix, and each block entry $M_i$ is irreducible, $i = 1, \ldots, t$.

\emph{Step 2:} We define a new matrix
$
    \hat{M} = M + \xi I_n,
$
where {$\xi= \max_{i=1,\dots,n} |m_{i i}|$, which implies that $\hat{M}$ has only nonnegative entries. Furthermore, since $M \textbf{1} = 0$, where $\textbf{1}$ is a vector of ones, we have that $\hat M \textbf{1} = \xi \textbf{1}$.}
According to Step 1, we have either $\hat{M}$ is irreducible, or
$$
    P\hat{M}P^{-1} = \begin{bmatrix} \hat{M}_1 & & \\ & \ddots & \\ & & \hat{M}_t
    \end{bmatrix},
$$
where $\hat{M}_i, i = 1, \ldots, t$ are irreducible non-negative matrices. {Note that $\textbf{1}$ is permutation invariant (i.e., $P\textbf{1}=P^{-1}\textbf{1}=\textbf{1}$), $\hat M_i \textbf{1} = \xi \textbf{1}$ and $\hat M_i$ is non-negative and irreducible for each $i$.
{As $\textbf{1}$ is a positive eigen-vector, }according to Perron-Frobenius theorem, $\xi$ is the spectral radius and the eigenvalue of $\hat M_i$ for each $i$. This implies that $\xi$ is also the spectral radius and an eigenvalue for the matrix $\hat M$.}

\emph{Step 3}: According to the Perron Frobenius theorem, $\xi$ is the spectral radius of $\hat{M}_i$, and the corresponding left eigen-vectors can be chosen to be positive element-wise. Thus, by stacking the positive left eigen-vectors of $\hat{M}_i$, there exists $$d = \begin{bmatrix} d_1^2 & d_2^2 & \ldots & d_n^2 \end{bmatrix}^\tr$$
with positive scalars $d_i, i = 1, \ldots, n$ such that
$
    d^\tr \hat{M} = d^\tr \xi,
$
leading to
$
    d^\tr {M} = 0,
$
which {implies}~\eqref{eq:SDD_s0}. This completes the proof.
\end{proof}

Now are ready to prove Corollary~\ref{prop:fwsdd}.

 $1 \Leftrightarrow 2$ simply follows Theorem~\ref{theo:blockfw} with $\alpha = \{1, 1, \ldots, 1\}$. We now prove $2 \Leftrightarrow 3$.

    $2 \Rightarrow 3:$   We first define a matrix $M = [m_{ij}] \in \mathbb{R}^{n \times n}$ as in~\eqref{eq:M}. Without loss of generality, we can assume $M$ has a symmetric nonzero pattern. This is because the constraint~\eqref{eq:fw2_s2} allows us to set $z_{ji} = 0$ whenever $z_{ij} = 0$.
    According to Lemma~\ref{lem:sdd}, there exist positive scalars $d_i > 0$ such that
    \begin{equation} 
        |m_{ii}|d_i^2 \geq \sum_{j=1, j \neq i}^n \left(m_{ji}d_j^2\right), \quad \forall i = 1, \ldots, n.
    \end{equation}

    Note that
    \begin{equation} \label{eq:squareineqaulity}
        x + y \geq 2 \sqrt{xy}, \forall x\geq 0, y \geq 0.
    \end{equation}
    Thus, we have that
        \begin{align}
                \sum_{j = 1, j \neq i}^n  \left(\frac{d_j}{d_i}|a_{ij}|\right) &\leq \sum_{j = 1, j \neq i}^n  \sqrt{m_{ji}\frac{d^2_j}{d^2_i}m_{ij}}  \nonumber\\
                & \leq \frac{1}{2} \sum_{j = 1, j \neq i}^n  \left(m_{ji}\frac{d^2_j}{d^2_i} + m_{ij}\right) \nonumber\\
                & = \frac{1}{2} \sum_{j = 1, j \neq i}^n m_{ij} + \frac{1}{2d_i^2}\sum_{j = 1, j \neq i}^n m_{ji}d_j^2 \nonumber\\
                & \leq \frac{1}{2} |m_{ii}| + \frac{1}{2d_i^2}|m_{ii}|d_i^2 \nonumber\\
                & \leq a_{ii}. \label{eq:SDD_s1}
        \end{align}
    In~\eqref{eq:SDD_s1}, the first inequality comes from~\eqref{eq:fw2_s2}, the second inequality is the fact~\eqref{eq:squareineqaulity}, the second to last inequality is from~\eqref{eq:SDD_s0}, and the last inequality comes from~\eqref{eq:fw2_s1}.
    Thus, $DAD$ is diagonally dominant with $D = \text{diag}(d_1, d_2, \ldots, d_n),$ \emph{i.e.}, $A \in \SDD_n$.

    $3 \Rightarrow 2:$ Suppose $A \in \SDD_n$. By definition, there exist positive $d_i$, such that
    $
        d_ia_{ii} \geq  \sum_{j = 1, j \neq i}^n |a_{ij}|d_j, i = 1, \ldots, n.
    $
    Now we choose
    $
        z_{ij} = \frac{d_j}{d_i}|a_{ij}| \geq 0, \forall i, j =1, \ldots, n, i \neq j,
    $
    which naturally satisfy the conditions in~\eqref{eq:fw2_s1}-\eqref{eq:fw2_s3}.
}

\subsection{Proof of Proposition~\ref{prop:upperbound}}
\label{app:C}

To facilitate our analysis, we define the distance between $(\mathcal{FW}_{\alpha,2}^n)^*$ and $\mathcal{FW}_{\alpha,2}^n$
 as the largest distance between a unit-norm matrix $M$ in $(\mathcal{FW}_{\alpha,2}^n)^*$  and the cone $\mathcal{FW}_{\alpha,2}^n$:
    $$
    \begin{aligned}
        &\text{dist}((\mathcal{FW}_{\alpha,2}^n)^*,\mathcal{FW}_{\alpha,2}^n) \\
        := &\sup_{M \in (\mathcal{FW}_{\alpha,2}^n)^*, \|M\|_F = 1} \, \text{dist}(M,\mathcal{FW}_{\alpha,2}^n\textbf{}).
    \end{aligned}
    $$
By definition, it is easy to see that
\begin{equation} \label{eq:norminequality}
\begin{aligned}
    \text{dist}((\mathcal{FW}_{\alpha,2}^n)^*,\mathcal{FW}_{\alpha,2}^n) \;&\geq\; \text{dist}((\mathcal{FW}_{\alpha,2}^n)^*,\mathbb{S}^n_+),  \\
    \text{dist}((\mathcal{FW}_{\alpha,2}^n)^*,\mathcal{FW}_{\alpha,2}^n) \;&\geq\; \text{dist}(\mathbb{S}^n_+,\mathcal{FW}_{\alpha,2}^n).
\end{aligned}
\end{equation}

The following strategy is motivated by~\cite{blekherman2020sparse}. Consider a matrix $M \in (\mathcal{FW}_{\alpha,2}^n)^*$ with $\|M\|_F = 1$. We construct a new $\hat{M} \in \mathcal{FW}_{\alpha,2}^n$. By definition, the principal block submatrices of $M$ are
$
    M^{ij} := E^{\alpha}_{ij} M (E^{\alpha}_{ij})^\tr \succeq 0, \quad 1 \leq i <j \leq p.
$
Then, we construct a new matrix as
$$
    \hat{M}:=  \sum_{1 \leq i <j \leq p} (E^{\alpha}_{ij})^\tr M^{ij} E^{\alpha}_{ij},
$$
for which we have $\hat{M} \in \mathcal{FW}_{\alpha,2}^n$. Moreover, notice that the block entries of $\hat{M}$ are just scaling of the block entries of $M$.
\begin{enumerate}
    \item \textit{Diagonal blocks:} these are scaled by the factor $p-1$, \emph{i.e.}
    $
        \hat{M}_{ii} = (p-1)M_{ii}, \quad i = 1, \ldots, p,
    $ due to the construction of the basis matrices $E^\alpha_{ij}$.
    \item \textit{Off-diagonal blocks:} the off-diagonal blocks are the same, \emph{i.e.},
    $
        \hat{M}_{ij} = M_{ij}, \quad  i \neq j.
    $

    \end{enumerate}

To even out these factors, we consider $\frac{2}{p}\hat{M}$, leading to
$$
    \begin{aligned}
        (M - \frac{2}{p}\hat{M})_{ii} &= \frac{2-p}{p}M_{ii},  \quad  i = 1, \ldots,p \\
        (M - \frac{2}{p}\hat{M})_{ij} &= \frac{p-2}{p}M_{ij}, \quad i \neq j.
    \end{aligned}
$$
Therefore, we have
$$
    \text{dist}(M,\mathcal{FW}_{\alpha,2}^n) \leq \left\|M - \frac{2}{p}\hat{M}\right\|_F = \frac{p-2}{p}\|M\|_F = \frac{p-2}{p}.
$$
Since this holds for all unit-norm matrices in $(\mathcal{FW}_{\alpha,2}^n)^*$, this upper bound holds for  $ \text{dist}((\mathcal{FW}_{\alpha,2}^n)^*,\mathcal{FW}_{\alpha,2}^n)$. Considering the inequality~\eqref{eq:norminequality},  we complete the proof of Proposition~\ref{prop:upperbound}.

\subsection{Proof of Proposition~\ref{prop:lowerbound}}
\label{app:D}

The main idea in our proof follows~\cite[Theorem 3]{blekherman2020sparse}, which constructs a special family of matrices $G(a,b,n)$ that admits an easier computation for the distance between $G(a,b,n)$ and $\mathcal{S}^n_+$. This distance can be used as a lower bound in our purpose. We mainly adapt this idea to the case of block factor-width-two matrices.
Similar to~\cite{blekherman2020sparse}, we consider a family of matrices
\begin{equation} \label{eq:Gmatrix}
    G(a,b,n) := (a+b)I_n - a \textbf{1}\textbf{1}^\tr,
\end{equation}
where $I_n$ is the $n \times n$ identity matrix, and $\textbf{1}$ is the column vector with all entries being 1. We have the following fact:
\begin{lemma}[\!\cite{blekherman2020sparse}] \label{lemma:Gmatrix}
Consider the family of matrices~\eqref{eq:Gmatrix}.
\begin{itemize}
    \item The eigenvalues of $G(a,b,n)$ are $b - (n-1)a$ with multiplicity 1, and $b + a$ with multiplicity $n-1$.
 \item If $a \geq 0, b \geq 0$, then
 $
    \text{dist}(G(a,b,n),\mathbb{S}^n_+) = \max\{(n-1)a -b, 0\}.
 $
\end{itemize}
\end{lemma}

Given a homogeneous partition $\alpha$, the block size is
$
    \alpha_{\max} = \frac{2n}{p}.
$
By Lemma~\ref{lemma:Gmatrix}, it is not difficult to see that if $a\geq 0, b \geq 0$, we have
$
    G(a,b,n) \in (\mathcal{FW}_{\alpha,2}^n)^*  \; \Leftrightarrow \;  b \geq (\alpha_{\max} - 1)a,
$
since every principle block matrix in $G(a,b,n)$ is PSD.
Let
$$
    \begin{aligned}
        \hat{a} &= \frac{1}{\sqrt{(\alpha_{\max}-1)^2n +n(n-1)}} \\
        \hat{b}& = (\alpha_{\max} - 1)\hat{a},
    \end{aligned}
$$
then it is easy to verify that $\|G(\hat{a},\hat{b},n)\|_F = 1$ and $G(\hat{a},\hat{b},n) \in (\mathcal{FW}_{\alpha,2}^n)^*$. Therefore
$$
\begin{aligned}
    \text{dist}((\mathcal{FW}_{\alpha,2}^n)^*,\mathbb{S}^n_+) &\geq  \text{dist}(G(\hat{a},\hat{b},n) ,\mathbb{S}^n_+) \\
    &= (n-1)\hat{a} - (\alpha_{\max}-1)\hat{a} \\
    &= (n-\alpha_{\max})\hat{a}.
\end{aligned}
$$
Then we have
$$
    \begin{aligned}
        \text{dist}((\mathcal{FW}_{\alpha,2}^n)^*,\mathbb{S}^n_+) &\geq (n-\alpha_{\max})\hat{a} = \left(n-\frac{2n}{p}\right)\hat{a} \\
       &= \frac{1}{\sqrt{\frac{4n}{p^2} - \frac{4}{p}+1}}\frac{p-2}{p}.
    \end{aligned}
$$
This completes the proof.

\balance

\bibliographystyle{IEEEtran}        
\bibliography{references}

%



\end{document}